\documentclass[12pt]{amsart}
\usepackage{amsmath,amssymb}
\usepackage{color}
\usepackage{amsmath,amssymb}
\usepackage[top=2cm, bottom=2cm, left=2cm, right=2cm]{geometry}
\usepackage{amsthm}
\usepackage{type1cm}
\usepackage{bm}
\usepackage{eucal}
\usepackage{amscd} 
\usepackage{graphicx,xcolor}
\usepackage{threeparttable}
\allowdisplaybreaks
\title[Classification of SNC-algebras in dimension five]
{Classification of SNC-algebras in dimension five}
\author{Haruka Sugai}
\address{Department of Mathematics, Faculty of Science and Technology,
Tokyo University of Science,
Noda, Chiba, 278-8510 Japan}
\email{6125702@ed.tus.ac.jp}

\date{}
\theoremstyle{plain} 

\newtheorem{thm}{Theorem}[section]
\newtheorem{lem}[thm]{Lemma}
\newtheorem{defi}[thm]{Definition}
\newtheorem{prop}[thm]{Proposition}
\newtheorem{itr}[thm]{Example}
\newtheorem{remk}[thm]{Remark}

\newcommand{\relmiddle}[1]{\mathrel{}\middle#1\mathrel{}}
\thanks{
\hspace{-5mm}{\it $2020$ Mathematics Subject Classification.} 53C30.\\
{\it Key words and phrases.} Homogeneous manifold, Lie algebra, negative curvature, Ricci curvature.}

\begin{document}
\maketitle
\begin{quote}
{\sc Abstract.}
 Every homogeneous manifold of negative curvature is
known to be isometric to a Lie group with a left invariant metric. We define an
SNC-algebra to be a Lie algebra which admits an inner product of strictly negative
curvature. In the author's joint paper in 2022, we classified SNC-algebras in dimension four. In this article, we classify SNC-algebras in dimension five, as well as we calculate Ricci curvature of SNC-algebras in dimension four.
\end{quote}
\section{Introduction}
The subject of this paper is the classification of connected homogeneous manifolds with negative curvature. 
Wolf\,\cite{w} proved that a connected homogeneous manifold with non-positive curvature admits a transitive solvable group of isometries, that is, admits a transitive action of a connected closed solvable subgroup of the group of isometries. In addition, it is known that connected homogeneous manifolds with negative curvature are simply connected by Kobayashi\,\cite{k}.  When studying connected homogeneous manifolds with negative curvature, it is sufficient to examine solvable Lie groups are Lie groups endowed with left invariant metric. Since a left invariant metric on a Lie group $G$ is determined by an inner product on its Lie algebra $\mathfrak{g}$, we can consider the curvature for  the Lie algebra with an inner product. We say that $\mathfrak{g}$ has an inner product of strictly negative curvature if the corresponding homogeneous manifold has a negative curvature. We defined a special class of Lie algebras in \cite{gs}. A Lie algebra $\mathfrak{g}$ is called an SNC-algebra, that is, a Lie algebra of strictly negative curvature type, provided that $\mathfrak{g}$ has an inner product of strictly negative curvature.
Furthermore, according to Heintze\,\cite{h}, in classifying $n$-dimensional SNC-algebras, it is sufficient to consider $(n-1)$-dimensional nilpotent Lie algebras and all the derivations associated with these Lie algebras.
In addition, the structure of nilpotent Lie algebras are known by Magnin\,\cite{mag} when the dimension is less than seven. Therefore, we will consider the derivations associated with each nilpotent Lie algebra.

 Furthermore, according to Heintze\,\cite{h}, the condition that a  symmetric space with negative curvature must satisfy are indicated. Since a symmetric space with negative curvature is also an SNC-algebra, we determine the SNC-algebras which admit structures of symmetric spaces. We note that among such SNC-algebras there exist SNC-algebras which are not isomorphic to each other as Lie algebras but their curvatures with respective to the uniquely determined orthnormal basis coincide.

The present paper is organized as follows. In Section 2, we introduce the results of Heintze\cite{h} and define an SNC-algebra. In Section 3, we introduce the results of classification in SNC-algebras of dimension four conducted in \cite{gs}. In Section 4, we classify five-dimensional SNC-algebras. In Section 5, we calculate Ricci curvature tensor of SNC-algebras in dimension four and determine Einstein cases. In Section 6, we analyze the details of  four and five-dimensional SNC-algebras which admit structures of symmetric spaces.

\setcounter{section}{1}
\section{Homogeneous manifolds with negative curvature}
Wolf \cite{w} proved that a connected homogeneous manifold with non-positive curvature admits a transitive solvable group of isometries. Furthermore, such a transitive solvable group is a Lie group endowed with a left invariant metric. From the above, Heintze showed the following theorem.

\begin{prop}[\cite{h} Proposition 1]
Let $M$ be a homogeneous manifold of non-positive curvature. Then $M$ is isometric to a connected solvable Lie group $G$ endowed with a left invariant metric.
\end{prop}
From this proposition, when considering the classification of connected homogeneous manifolds with negative curvature, it is sufficient to consider solvable Lie algebras with solvable Lie groups and left-invariant metrics.
\begin{defi}[\cite{gs} Definition 1.1]
A Lie algebra $\mathfrak{g}$ is called an SNC-algebra (a Lie algebra of strictly negative curvature type), provided that $\mathfrak{g}$ has an inner product of strictly negative curvature.
\end{defi}
Since a left invariant metric on a Lie group $G$ is determined by an inner product on its Lie algebra $\mathfrak{g}$, we can consider the curvature for  the Lie algebra with an inner product. We defined a special class of Lie algebras.
It is known that if a Lie algebra $\mathfrak{g}$ has an inner product of negative curvature, then $\mathfrak{g}$ is solvable. Therefore, SNC-algebras are solvable.
Next, we introduce the Heintze\cite{h} theorem, which has obtained the necessary and sufficient conditions for the  Lie algebra to become an SNC-algebra.

\begin{thm}[cf.\cite{h} Theorem 3]\label{thm:h3}
Let $\mathfrak{g}$ be a solvable Lie algebra. Then the following conditions are mutually equivalent. 
\begin{enumerate}
\item[(1)]$\mathfrak{g}$ is an SNC-algebra.
\item[(2)]$\mathfrak{g}'=[\mathfrak{g},\mathfrak{g}]$ stisfies the following two conditions.
\begin{enumerate}
\item[(2.1)]$\dim\mathfrak{g}'=\dim\mathfrak{g}-1$.
\item[(2.2)] There exists an element $A\in \mathfrak{g}$, which orthogonal to $\mathfrak{g}'$ with the property that all eigenvalues of the endomorphism $\mathrm{ad}A|\mathfrak{g}'$ of $\mathfrak{g}'$ have positive real parts.
\end{enumerate}
\end{enumerate}
\end{thm}

\begin{itr}[Milnor's example cf.\cite{m}]
Suppose that the Lie algebra $\mathfrak{g}$ whose $\dim \mathfrak{g}\ge 2$ has the property that the bracket product $[x,y]$ is always equal to a linear combination of $x$ and $y$. Then,
\[
[x,y]=\ell(x)y-\ell(y)x,
\]
where $\ell$ is a well defined linear mapping from $\mathfrak{g}$ to the real numbers. Choosing any positive definite metric, the sectional curvature $K$ satisfies the following.\ 
\[
K=-||\ell||^2.
\]

\end{itr}
Here, we verify that Milnor's example satisfies condition (2) of Theorem \ref{thm:h3}.
Let it be $\dim \mathfrak{g}=n$. Since, linear map $\ell$ satisfies the following.\ 
\[
\ell([X,Y])=\ell(\ell(X)Y-\ell(Y)X)=\ell(X)\ell(Y)-\ell(Y)\ell(X)=0.
\]
Therefore, $\mathfrak{g}'\subset \mathrm{Ker}(\ell)$.\ 

Conversely, if $X\in \mathrm{Ker}(\ell)$ and $Y\in \mathfrak{g}\backslash \mathfrak{g}'$,\ 
\[
[X,Y]=\ell(X)Y-\ell(Y)X=-\ell(Y)X.
\]
Then, $X=-\frac{[X,Y]}{\ell(Y)} \in \mathfrak{g}'$.\ Therefore, $\mathfrak{g}'= \mathrm{Ker}(\ell)$.\ 
$\dim \mathfrak{g}'=\dim \mathfrak{g}-1$ follows from $\dim \mathrm{Ker}(\ell)=n-1$.\ 

Next, let $A_0\in \mathfrak{g}$ that becomes $A_0\notin \mathfrak{g}'$.\ Therefore,\ $\mathrm{ad}A_0|_{\mathfrak{g}'}:{\mathfrak{g}'}\rightarrow {\mathfrak{g}'}$ can be regarded as\ 
\[
\mathrm{ad}A_0|_{\mathfrak{g}'}:\mathrm{Ker}(\ell)\rightarrow \mathrm{Ker}(\ell) . 
\]
Let $X$ be an arbitrary element of $\mathfrak{g}'$.\ Therefore,\ 
\[
\mathrm{ad}A_0|_{\mathfrak{g}'}(X)=\ell(A_0)X .
\]
\ $\mathrm{ad}A_0|_{\mathfrak{g}'}$ is a scalar transformation and \ $\ell(A_0)$ is an eigenvalue.\ If the real part of $\ell(A_0)$ is negative, $-A_0$ can retake so that the real part of $\ell(-A_0)$ is positive. From the above, Milner's example satisfies Theorem \ref{thm:h3}(2), so $\mathfrak{g}$ is an SNC-algebra. 
\section{Classification of SNC-algebras in dimension four}
In this chapter, we refer to the results of \cite{gs} which gave the classification of SNC-algebras in dimension four.  

Let $\mathfrak{n}$ be a Lie algebra, $\delta(\mathfrak{n})$ be the set of all derivations on $\mathfrak{n}$, and
$\delta^+(\mathfrak{n})$ be the set of elements of $\delta(\mathfrak{n})$ of which eigenvalues have positive real parts.
\begin{defi}[cf. \cite{gs}]
Let $D\in \delta(\mathfrak{n})$. Take an element $A_D$ outside $\mathfrak{n}$. We set $\mathfrak{n}(D)=\mathfrak{n}\oplus\mathbb{R}A_D$ as a vector
 space and give it the Lie algebra structure by setting 
\[
[A_D,X]:=D(X),~X\in \mathfrak{n}.
\]

The Lie algebra $\mathfrak{n}(D)$ is called the expanded Lie algebra by derivation $D$.
 \end{defi}
Next, we show conditions for isomorphic expanded Lie algebras.
$\text{End}(\mathfrak{n})$ is the group of all endomorphisms on $\mathfrak{n}$. 
$\text{Aut}(\mathfrak{n})$ is the group of all automorphisms on $\mathfrak{n}$. 
The following theorem holds. 
 \begin{thm}[cf. \cite{sh}]
Let $D_1,D_2\in \delta(\mathfrak{n})$. The following are equivalent.\ 
\begin{enumerate}
\item[(1)]$\mathfrak{n}(D_1) \cong \mathfrak{n}(D_2)$ (as Lie algebras).
\item[(2)] $ D_1=g^{-1}\mathrm{ad}(X)g+\lambda g^{-1}D_2g$ for some $g\in \mathrm{Aut}(\mathfrak{n}),\ X \in \mathfrak{n},$ and $\lambda \neq 0$.
\end{enumerate}
\end{thm}
When $D_1,D_2\in \delta(\mathfrak{n})$ satisfy (2), $D_1$ and $D_2$ are called O-equivalent.

Since \cite{sh} is written in Japanese, the proof of Theorem 3.2 is provided here.
\begin{proof}
First, we assume (1). Then, there exists an isomorphism 
 \[
 h:\mathfrak{n}(D_1) \rightarrow \mathfrak{n}(D_2) ,
 \]
where $\mathfrak{n}(D_i)=\mathfrak{n}\oplus\mathbb{R}A_{D_i}\ (i=1,2)$.
The restriction of $h$ to $\mathfrak{n}$ induces an automorphism on $\mathfrak{n}$.
Moreover, $h$ can be expressed as 
 \[
 h(A_{D_1})=X+\lambda A_{D_2}
 \]
by a certain $X\in \mathfrak{n}$ and $\ \lambda \neq 0$. Since
\begin{align*}
 h([A_{D_1},Y])
 &=[h(A_{D_1}),h(Y)]\\
 &=[X+\lambda A_{D_2},h(Y)]\\
 &=[X,h(Y)]+[\lambda A_{D_2},h(Y)]\\
 &=\mathrm{ad}(X)h(Y)+\lambda D_2(h(Y))
 \end{align*}
holds for $Y\in \mathfrak{n}$, we have 

\[
 h(D_1(Y))=\mathrm{ad}(X)h(Y)+\lambda D_2(h(Y)).
\]
Hence,
\[
  D_1(Y)=h^{-1}(\mathrm{ad}(X)h(Y))+\lambda h^{-1} (D_2(h(Y))) .
\]

Moreover, setting $g=h|_\mathfrak{n}$, we have
 \[
 D_1=g^{-1}\mathrm{ad}(X)g+\lambda g^{-1}D_2g,
 \]
 so (2) is satisfied. Next,   we assume (2). Then, 
\[
 D_1=g^{-1}\mathrm{ad}(X)g+\lambda g^{-1}D_2g
 \]
for $g\in \mathrm{Aut}(\mathfrak{n}),X \in \mathfrak{n}$, and $\lambda \neq 0$. We define a linear map $h:\mathfrak{n}(D_1) \rightarrow \mathfrak{n}(D_2)$ by 

 \[
 h(Y)= \begin{cases}
 g(Y) & (Y\in \mathfrak{n}),\\
 X+\lambda A_{D_2} & (Y=A_{D_1}).
 \end{cases}
 \]
Then, $h$ is a bijection.

 Next, we verify that $h$ is a homomorphism. By the definition of $h$, we have
\[
 h([Y,Z])=g([Y,Z])=[g(Y),g(Z)]=[h(Y),h(Z)].
\]
for $Y,Z\in \mathfrak{n}$. Moreover, 
 \begin{align*}
 h([A_{D_1},Z])
 &=h(D_1(Z))\\
 &=g(D_1(Z))\\
 &=\mathrm{ad}(X)g(Z)+\lambda D_2g(Z)\\
 &=[X,g(Z)]+\lambda D_2g(Z)\\
 &=[X+\lambda A_{D_2},g(Z)]\\
 &=[h(A_{D_1}),h(Z)].
 \end{align*}
Therefore, $\mathfrak{n}(D_1) \cong \mathfrak{n}(D_2)$.

\end{proof}
When $D_1,D_2\in \delta(\mathfrak{n})$ are O-equivalent, we write  $D_1\stackrel{O}{\sim} D_2$.
\begin{lem}[cf.\cite{sh}]
The relation $D_1\stackrel{O}{\sim} D_2$ is an equivalence relation on $\delta(\mathfrak{n})$.

 \end{lem}
\begin{proof}
If we assume $D_1,D_2,D_3\in \delta(\mathfrak{n}), g=id\in \mathrm{Aut}(\mathfrak{n})$, and $\lambda =1$, the reflexive low holds from

  \[
  D_1=\lambda g^{-1}D_1g=D_1.
  \]
 Next, we assume $D_1\stackrel{O}{\sim} D_2$. The following holds for a certain 

 \[
  D_1=g^{-1}\mathrm{ad}(X)g+\lambda g^{-1}D_2g.
 \]
Then,\ 
\begin{align*}
D_2
&=\frac{1}{\lambda}\mathrm{ad}(X)+\frac{1}{\lambda}gD_1g^{-1}\\
&=\frac{1}{\lambda}gg^{-1}\mathrm{ad}(X)gg^{-1}+\frac{1}{\lambda}gD_1g^{-1}.
\end{align*}
Furthermore, for any $Y\in \mathfrak{n}$, 
\[
  \frac{1}{\lambda}g^{-1}\mathrm{ad}(X)g=\mathrm{ad}(\frac{1}{\lambda}g^{-1}X)
\]
holds from
\begin{align*}
  \frac{1}{\lambda}g^{-1}\mathrm{ad}(X)gY
  &= \frac{1}{\lambda}g^{-1}([X,gY])\\
  &=[\frac{1}{\lambda}g^{-1}X,Y].
  \end{align*}

If $h=g^{-1},\ $and $ \frac{1}{\lambda}hX=Z$, the symmetric low holds from 
  \[
  D_2=h^{-1}\mathrm{ad}(Z)h+\frac{1}{\lambda}h^{-1}D_1h.
  \]

 Next, we assume $D_1\stackrel{O}{\sim} D_2,D_2\stackrel{O}{\sim} D_3$. For $g_1,g_2\in \mathrm{Aut}(\mathfrak{n})$, and $\lambda_1,\lambda_2 \neq 0$, the following holds.

\[
 D_1=g^{-1}\mathrm{ad}(X)g+\lambda g^{-1}D_2g,\ D_2=h^{-1}\mathrm{ad}(Y)h+\mu h^{-1}D_3h.
 \]
To summarize,

 \[
D_1=g^{-1}\mathrm{ad}(X)g+\lambda g^{-1}h^{-1}\mathrm{ad}(Y)hg+\lambda g^{-1}\mu h^{-1}D_3hg.
\]
Here,  

\[
D_1=g^{-1}h^{-1}\mathrm{ad}(hX+\lambda Y)hg+\lambda \mu g^{-1}h^{-1}D_3hg
\]
follows from 
\begin{align*}
  g^{-1}\mathrm{ad}(X)g
  &=g^{-1}h^{-1}h\mathrm{ad}(X)h^{-1}hg\\
  &=g^{-1}h^{-1}\mathrm{ad}(hX)hg.
  \end{align*}

From the above, the transitive law holds, and it becomes an equivalence relation.

\end{proof}
When $\mathfrak{n}$ is nilpotent and $D\in \delta^+(\mathfrak{n})$, $\mathfrak{n}(D)$ becomes an SNC-algebra from Theorem \ref{thm:h3}.

 On the contrary, when  $\mathfrak{g}$ is an SNC-algebra, $\mathfrak{g}$ is solvable, so $\mathfrak{n}=[\mathfrak{g},\mathfrak{g}]$ becomes a nilpotent Lie algebra. Furthermore, there exists $A\in\mathfrak{g}$ such that $\mathrm{ad}A|_{\mathfrak{n}}\in \delta^+(\mathfrak{n})$ by Theorem \ref{thm:h3}. Therefore, the following holds.

 \begin{thm}[cf.\cite{h}]\label{thm:douti}
Let $\mathfrak{g}$ be a Lie algebra of dimension $n$. Then the following are equivalent.
\begin{enumerate}
\item[(1)]$\mathfrak{g}$ is an SNC-algebra.
\item[(2)]There exist a nilpotent Lie algebra $\mathfrak{n}$ of dimension $n-1$ and $D\in\delta^+(\mathfrak{n})$ such that\ 
\[
\mathfrak{g}\cong  \mathfrak{n}(D).
\]
\end{enumerate}
\end{thm}
In order to classify SNC-algebras, it is sufficient to consider nilpotent Lie algebras  $\mathfrak{n}$ and $\delta^+(\mathfrak{n})$ by Theorem3.4.
Here, we introduce the classification results of SNC-algebras in dimension four conducted in \cite{gs,sh}. Let $\mathfrak{g}$ be an SNC-algebra in dimension four. There exists an expanded Lie algebra $\mathfrak{n}(D)$ that satisfying $\mathfrak{g}\cong  \mathfrak{n}(D)$ by Theorem \ref{thm:douti}. In particular, $\mathfrak{n}$ is a three-dimensional Lie algebra. It is known that $\mathfrak{n}$ is either (A) or (B) below.

    \begin{enumerate}
\item[(A)]$\mathfrak{n}$ is abelian.
\item[(B)]$\mathfrak{n}$ has an orthonormal basis $e_1,e_2,e_3$ with the relations
\[
[e_1,e_2]=e_3,\ [e_2,e_3]=[e_3,e_1]=0.
\]
\end{enumerate}
\subsection{Abelian case}
Assume that $\mathfrak{n}$ is abelian. Note that in this case $\delta(\mathfrak{n})=\mathrm{End}(\mathfrak{n})$, and$\ \mathrm{Aut}(\mathfrak{n})=\mathrm{GL}(\mathfrak{n})$ hold. Assuming $\mathfrak{n}(D_1) \cong \mathfrak{n}(D_2)$ with respect to $D_1,D_2\in \delta(\mathfrak{n})$, isomorphism $h:\mathfrak{n}(D_1) \rightarrow \mathfrak{n}(D_2)$ exists. When $h(A_{D_1})=X+\lambda A_{D_2},\ X\in \mathfrak{n},$ and$\ \lambda \neq 0$, 

  \begin{align*}
 hD_1(Y)
 &=h([A_{D_1},Y])\\
 &=[h(A_{D_1}),h(Y)]\\
 &=[X+\lambda A_{D_2},h(Y)]\\
 &=\lambda D_2h(Y),
 \end{align*}
holds for $Y\in \mathfrak{n}$. If $g=h|_{\mathfrak{n}}$, then $g\in \mathrm{Aut}(\mathfrak{n})$ and, $D_1=\lambda g^{-1}D_2g$ holds, so the following hold from Theorem \ref{thm:douti}.

  \begin{lem}[\cite{sh} Lemma 2.1]\label{lem:ab}
  Let $D_1,D_2\in \delta(\mathfrak{n})$. Then the following are equivalent.
  \begin{enumerate}
\item[(1)]$\mathfrak{n}(D_1) \cong \mathfrak{n}(D_2)$\ (as Lie algebras).
\item[(2)]$D_1=\lambda g^{-1}D_2g$ for some $g\in \mathrm{GL}(\mathfrak{n}),\ \lambda \neq 0$.
\end{enumerate}
\end{lem}
When $D_1,D_2\in \delta(\mathfrak{n})$ satisfy (2), $D_1$ and $D_2$ are called {\it A-equivalent}. 
Let $e_1,e_2,e_3$ be an arbitrarily orthonormal basis of $\mathfrak{n}$.\ 

From the discussion of Lemma \ref{lem:ab} and the Jordan canonical form, the following holds.

\begin{prop}[cf.\cite{gs}]\label{prop:3a}
The representation matrix of $D\in\delta^+(\mathfrak{n})$ concerning with $e_1,e_2,e_3$ is conjugate to one of the following.

\[
 \left(
  \begin{array}{ccc}
     x & 0 & 0 \\
     0 & y & 0 \\
     0 & 0 & 1
  \end{array}
 \right),\ 
 \left(
  \begin{array}{ccc}
    z & 0 & 0 \\
    0 & 1 & 1 \\
    0 & 0 & 1
  \end{array}
 \right),\ 
 \left(
  \begin{array}{ccc}
    1 & 1 & 0 \\
    0 & 1 & 1 \\
    0 & 0 & 1
  \end{array}
 \right),\ 
 \left(
  \begin{array}{ccc}
   \alpha & -\beta & 0\\
   \beta &\alpha & 0 \\
    0 & 0 & 1
  \end{array}
 \right),
\]
where $0<x\le y\le 1,\ 0<z,\ 0<\alpha,\beta$.

\end{prop}
Proposition \ref{prop:3a} implies the following classification result immediately.

\begin{thm}[\cite{gs} Theorem 2.2]\label{thm:ab}
Let $\mathfrak{g}$ be an SNC-algebra of dimension four with the abelian derived algebra $\mathfrak{n}=[\mathfrak{g},\mathfrak{g}]$, and let $e_1,e_2,e_3$ be an arbitrarily orthonormal basis of $\mathfrak{n}$. Then there exists \ $e_4\in \mathfrak{g}\backslash\mathfrak{n}$ whose bracket operations are given by one of the following.\ 
\begin{enumerate}
\item[(1)]
    $ [e_4,e_1]=xe_1,\ [e_4,e_2] =ye_2,\   [e_4,e_3]=e_3,\ 0<x\le y\le 1$.

\item[(2)]
   $ [e_4,e_1]=ze_1,\ [e_4,e_2] =e_2,\   [e_4,e_3]=e_2+e_3,\ 0<z$.
\item[(3)]
   $ [e_4,e_1]=e_1,\ [e_4,e_2] =e_1+e_2,\   [e_4,e_3]=e_2+e_3$.
\item[(4)]
 $ [e_4,e_1]=\alpha e_1+\beta e_2,\ [e_4,e_2] =-\beta e_1+\alpha e_2,\ [e_4,e_3]=e_3,\ 0<\alpha,\beta$.
 
\end{enumerate}
Any two Lie algebras in different types or with different parameters are not isomorphic to each other.

\end{thm}
\begin{remk}
Assume that four-dimensional Lie algebra satisfies the condition Milnor indicated by Example 2.1. From the conditions, $\mathfrak{n}=[\mathfrak{g},\mathfrak{g}]$ is abelian. For an orthonormal basis $e_1,e_2,e_3$ on $\mathfrak{n}$ fixed arbitrarily, define $e_4\in \mathfrak{g}$ and linear map $\ell :\mathfrak{g}\to \mathbb{R}$ satisfies 

\[
\ell(e_4)=1, \ell(e_1)=\ell(e_2)=\ell(e_3)=0.
\]
Then, the following is satisfied.

\begin{align*}
[e_4,e_1]=&\ell(e_4)e_1-\ell(e_1)e_4=e_1,\\
[e_4,e_2]=&\ell(e_4)e_2-\ell(e_2)e_4=e_2,\\
[e_4,e_3]=&\ell(e_4)e_3-\ell(e_3)e_4=e_3.
\end{align*}
From the above, Milnor's example is the case of $x=y=1$ in Theorem \ref{thm:ab} (1). 

\end{remk}
\subsection{Non abelian case}

Let $\mathfrak{n}$ be a non-abelian nilpotent Lie algebra of dimension three. Then $\mathfrak{n}$ has the orthonormal basis $e_1,e_2,e_3$ with the relations
\[
[e_1,e_2]=e_3,\ [e_2,e_3]=[e_3,e_1]=0.
\]

Let $\mathfrak{m}$ be a linear subspace of $\mathfrak{n}$ spanned by $e_1,e_2$, and let $\pi$ be an orthgonal projection from $\mathfrak{n}$ onto $\mathfrak{m}$. First, we examine the structures of $\delta^+(\mathfrak{n})$ and $\mathrm{Aut}(\mathfrak{n})$.

\begin{lem}[\cite{gs} Lemma3.1]\label{lem:n3} 

\begin{enumerate}
\item[(1)]
The representation matrix of $ D\in \delta^+(\mathfrak{n})$ concerning with $e_1,e_2,e_3$ takes the following form.
 \begin{align}
  D(e_1,e_2,e_3)=
  (e_1,e_2,e_3)\left(
  \begin{array}{ccc}
     x_{11} &x_{12} & 0 \\
     x_{21} &x_{22} & 0 \\
     x_{31} &x_{32} &x_{11}+x_{22}
  \end{array}
 \right). \tag{3.1}
 \end{align}
 On the contrary, if the representation matrix of $ D\in \mathrm{End}(\mathfrak{n})$ concerning with $e_1,e_2,e_3$ satisfies Equation (3.1),\ each eigenvalue of $D_\mathfrak{m}:=\pi\circ D|_{\mathfrak{m}}$ has a positive real part, and $\mathrm{tr}(D_\mathfrak{m})>0$,  then $D\in \delta^+(\mathfrak{n})$.
 \item[(2)]
 The representation matrix of $ A\in \mathrm{Aut}(\mathfrak{n})$ concerning with $e_1,e_2,e_3$ takes the following form.
 \begin{align}
 A(e_1,e_2,e_3)=
  (e_1,e_2,e_3)\left(
  \begin{array}{ccc}
     y_{11} &y_{12} & 0 \\
     y_{21} &y_{22} & 0 \\
     y_{31} &y_{32} &y_{11}y_{22}-y_{12}y_{21}
  \end{array}
  \right).\tag{3.2}
 \end{align}
On the contrary, if the representation matrix of $ A\in \mathrm{GL}(\mathfrak{n})  $ concerning with $e_1,e_2,e_3$ satisfies Equation (3.2), then $A\in \mathrm{Aut}(\mathfrak{n})$.
\end{enumerate}
  \end{lem}
 Furthermore, the following lemma holds.

 \begin{lem}[cf. \cite{sh}]\label{lem:n3-1}
For any $D\in \delta^+(\mathfrak{n})$, there exists unique $\Tilde{D}\in \delta^+(\mathfrak{n})$ that is conjugate to $D$ by $\mathrm{Aut}(\mathfrak{n})$ such that the representation matrix of $\Tilde{D}$ concerning with $e_1,e_2,e_3$ takes the following form.
 \[
  \Tilde{D}=
  \left(
  \begin{array}{ccc}
     x_{11} &x_{12} & 0 \\
     x_{21} &x_{22} & 0 \\
     0 &0 &x_{11}+x_{22}
  \end{array}
 \right) \tag{3.3}.
  \]
\end{lem}
\begin{proof}
We take an element of $\mathrm{Aut}(\mathfrak{n})$ whose representation matrix concerning with $e_1,e_2,e_3$ is

   \[A=
\left(
  \begin{array}{ccc}
     x_{11} &x_{12} & 0 \\
     x_{21} &x_{22} & 0 \\
     x_{11}x_{31}+x_{21}x_{32} &x_{22}x_{32}+x_{12}x_{31} & x_{11}x_{22}-x_{12}x_{21}
  \end{array}
 \right).
  \]
Since it becomes $AD=\Tilde{D}A$, $D$ and $\Tilde{D}$ are conjugate in $\mathrm{Aut}(\mathfrak{n})$. 
  \end{proof}
From now on, we identify $D\in \delta^+(\mathfrak{n})$ with $\Tilde{D}\in \delta^+(\mathfrak{n})$ in Lemma \ref{lem:n3-1}. Hence, we may assume the representation matrix with respect to $e_1,e_2,e_3$ is given by (3,3).
Next, a subset $\Tilde{\delta}^+(\mathfrak{n})$ of ${\delta}^+(\mathfrak{n})$ is defined by
\[
\Tilde{\delta}^+(\mathfrak{n}):=\left\{\left(
  \begin{array}{c|c}
    D_{\mathfrak{m}} &\vec{0} \\ \hline
   ^t\vec{0}&\mathrm{tr}(D_{\mathfrak{m}})  
   \end{array}
\right)\in {\delta}^+(\mathfrak{n})\relmiddle|
D_{\mathfrak{m}}\in \mathrm{GL}(\mathfrak{m})
\right\}.
 \]
Considering the quotient set of ${\delta}^+(\mathfrak{n})$ and $\Tilde{\delta}^+(\mathfrak{n})$ under O-equivalence, the following holds.
 \[
\Tilde{\delta}^+(\mathfrak{n})/\stackrel{O}{\sim}\ \subset {\delta}^+(\mathfrak{n})/\stackrel{O}{\sim}
\]
On the other hand, from Lemma \ref{lem:n3-1}, for any $D\in \delta^+(\mathfrak{n})$, there exists exactly one $\Tilde{D}\in \Tilde{\delta}^+(\mathfrak{n})$ that is O-equivalent. Then, if we denote the equivalence class of $D$ under O-equivalent by $[D]$ and the equivalence class of $\Tilde{D}$ under O-equivalet by $[\Tilde{D}]$, then since 
\[
[D]=[\Tilde{D}]\in \Tilde{\delta}^+(\mathfrak{n})/\stackrel{O}{\sim}
\]
holds, we can see that
 \[
\Tilde{\delta}^+(\mathfrak{n})/\stackrel{O}{\sim}\ = {\delta}^+(\mathfrak{n})/\stackrel{O}{\sim}
\]
follows. From the above, it is sufficient to consider O-equivalent on $\Tilde{\delta}^+(\mathfrak{n})$.\\
Next, the following commutative diagram holds for $D\in \Tilde{\delta}^+(\mathfrak{n}),\ g\in \mathrm{Aut}(\mathfrak{n}),\ g_m\in \mathrm{Aut}(\mathfrak{m}),\ $and $\pi:\mathfrak{n} \rightarrow \mathfrak{m}$.
   \[
  \begin{CD}
     \mathfrak{n} @>{D}>> \mathfrak{n} \\
  @VV{\pi}V    @VV{\pi}V \\
     \mathfrak{m}   @>>{D_\mathfrak{m}}>  \mathfrak{m}
  \end{CD}
 ~~~~~~~~~~~~~~~      \ \ 
    \begin{CD}
     \mathfrak{n} @>{g}>> \mathfrak{n} \\
  @VV{\pi}V    @VV{\pi}V \\
     \mathfrak{m}   @>>{g_\mathfrak{m}}>  \mathfrak{m}
  \end{CD}
\]
From the above, the following holds.
\begin{lem}[\cite{gs} Lemma 3.4]\label{lem:3b}
Let $D_1,D_2\in \Tilde{\delta}^+(\mathfrak{n})$, the following are equivalent.
\begin{enumerate}
\item[(1)]$\mathfrak{n}(D_1) \cong \mathfrak{n}(D_2)$\ (as Lie algebras).
\item[(2)]$(D_1)_\mathfrak{m}= \lambda{g_{\mathfrak{m}}}^{-1}(D_2)_\mathfrak{m}g_{\mathfrak{m}}$ for some $g_{\mathfrak{m}}\in \mathrm{GL}(\mathfrak{m}),\lambda \neq 0$.
\end{enumerate}
\end{lem}

When $D_1,D_2\in \Tilde{\delta}^+(\mathfrak{n})$ satisfy (2), $D_1$ and $D_2$ are called B-equivalent, we write $D_1\stackrel{B}{\sim} D_2$.
Furthermore, from Lemma \ref{lem:3b}, the following holds.
\[
\Tilde{\delta}^+(\mathfrak{n})/\stackrel{O}{\sim}\ = \Tilde{\delta}^+(\mathfrak{n})/\stackrel{B}{\sim}.
\]
When classifying four-dimensional SNC Lie algebras in the case where the derived algebra $\mathfrak{n}$ is non-aberian, it was shown that it is sufficient to consider the quotient set by B-equivalent over $\Tilde{\delta}^+(\mathfrak{n})$. Therefore, the following holds.
\begin{prop}[cf. \cite{gs}]\label{prop:3b}
The representation matrix of $D\in\delta^+(\mathfrak{n})$ concerning with $e_1,e_2,e_3$ is conjugate to one of the following.
\[
  \left(
  \begin{array}{ccc}
   1-x &0&0 \\
   0&x&0\\
   0&0&1
   \end{array}
   \right), 
 \left(
  \begin{array}{ccc}
   \frac{1}{2} &1&0 \\
   0&\frac{1}{2}&0\\
   0&0&1
   \end{array}
   \right), 
 \left(
  \begin{array}{ccc}
   \frac{1}{2}&-\alpha&0 \\
   \alpha&\frac{1}{2}&0\\
   0&0&1
   \end{array}
   \right),
\]
where $0<x< \frac{1}{2}, 0<\alpha$.
\end{prop}
From the above, SNC-algebras of (B) can be classified as follows.
   \begin{thm}[\cite{gs} Theorem 3.5]
Let $\mathfrak{g}$ be an SNC-algebra of dimension four with non-abelian derived algebra $\mathfrak{n}=[\mathfrak{g},\mathfrak{g}]$, which has an orthonormal basis $e_1,e_2,e_3$ with the relations
\[
[e_1,e_2]=e_3,\ [e_2,e_3]=[e_3,e_1]=0 .
\]
Then there exists $e_4\in \mathfrak{g}\backslash\mathfrak{n}$ whose bracket operations are given by one of the following.\ 
\begin{enumerate}
\item[(1)]
    $ [e_4,e_1]=(1-x)e_1,\ [e_4,e_2] =xe_2,\ [e_4,e_3]=e_3,\ 0<x< \frac{1}{2}$.

\item[(2)]
   $ [e_4,e_1]=\frac{1}{2}e_1,\ [e_4,e_2] =e_1+\frac{1}{2}e_2,\ [e_4,e_3]=e_3$.
\item[(3)]
   $ [e_4,e_1]=\frac{1}{2}e_1+\alpha e_2,\ [e_4,e_2] =-\alpha e_1+\frac{1}{2}e_2,\   [e_4,e_3]=e_3,\ 0<\alpha$.
 \end{enumerate}
 \end{thm}

\section{Classification of SNC-algebras in dimension five}
In this section, we classify SNC-algebras in dimension five. Let $\mathfrak{g}$ be an SNC-algebra in dimension five. From Theorem \ref{thm:douti}, there exists an expanded Lie algebra $\mathfrak{n}(D)$ that is isomorphic to $\mathfrak{g}$. Here, $\mathfrak{n}$ is a four-dimensional nilpotent Lie algebra, and from Magnin\cite{m}, the following three types exist.
    \begin{enumerate}
\item[(A)]$\mathfrak{n}$ is abelian.
\item[(B)]$\mathfrak{n}$ has a basis $e_1,e_2,e_3,e_4$ with the relations, 
\[
[e_1,e_2]=e_3,\ [e_1,e_3]=e_4,\ \text{others are zero.}
\]
\item[(C)]$\mathfrak{n}$ has a basis $e_1,e_2,e_3,e_4$ with the relations, 
\[
[e_1,e_2]=e_4,\ \text{others are zero.}
\]
\end{enumerate}
\subsection{Case of (A)}
In the case where $\mathfrak{n}$ is abelian, it can be classified immediately from Lemma \ref{lem:ab}.
\begin{prop}
The representation matrix of $D\in\delta^+(\mathfrak{n})$ concerning with $e_1,e_2,e_3,e_4$ is conjugate to one of the following.
\[
 \left(
  \begin{array}{cccc}
     x_1 & 0 & 0 & 0\\
     0 & x_2 & 0 & 0\\
     0 & 0 & x_3 & 0\\
     0 & 0 & 0 & 1\\  
  \end{array}
 \right),\ 
 \left(
  \begin{array}{cccc}
     y_1 & 0 & 0 & 0\\
     0 & y_2 & 0 & 0\\
     0 & 0 & 1 & 1\\
     0 & 0 & 0 & 1\\  
  \end{array}
 \right),\ 
 \left(
 \begin{array}{cccc}
     y_1 & 0 & 0 & 0\\
     0 & y_2 & 0 & 0\\
     0 & 0 &  1 & -\beta\\
     0 & 0 & \beta &1\\  
  \end{array}
 \right)\]\[
 \left(
  \begin{array}{cccc}
     y_1 & 0 & 0 & 0\\
     0 & 1 & 1 & 0\\
     0 & 0 & 1 & 1\\
     0 & 0 & 0 & 1\\  
  \end{array}
 \right),\ 
 \left(
  \begin{array}{cccc}
     y_1 & 1 & 0 & 0\\
     0 & y_1 & 0 & 0\\
     0 & 0 & 1 & 1\\
     0 & 0 & 0 & 1\\  
  \end{array}
 \right),\ 
 \left(
 \begin{array}{cccc}
     y_1 & 1 & 0 & 0\\
    0 &y_1 & 0 & 0\\
     0 & 0 &  1 & -\beta\\
     0 & 0 & \beta &1\\  
  \end{array}
 \right),\]\[
 \left(
 \begin{array}{cccc}
     \alpha & -\beta & 0 & 0\\
     \beta &\alpha & 0 & 0\\
     0 & 0 &  1 & -\beta'\\
     0 & 0 & \beta' &1\\  
  \end{array}
 \right),\ 
 \left(
  \begin{array}{cccc}
     1 & 1 & 0 & 0\\
     0 & 1 & 1 & 0\\
     0 & 0 & 1 & 1\\
     0 & 0 & 0 & 1\\  
  \end{array}
 \right), \left(
 \begin{array}{cccc}
     1 & -\beta & 1 & 0\\
     \beta &1 & 0 & 1\\
     0 & 0 &  1 & -\beta\\
     0 & 0 & \beta &1\\  
  \end{array}
 \right), 
\]
where,\ $0<x_1\le x_2\le x_3\le 1,\ 0<y_1\le y_2,\ 0<\alpha,\beta,\beta'$. 
\end{prop}
To summarize the above discussion,  the case of (A) can be classified as follows.
\begin{thm}
Let $\mathfrak{g}$ be an SNC-algebra of dimension five with abelian derived algebra $\mathfrak{n}=[\mathfrak{g},\mathfrak{g}]$, and let $e_1,e_2,e_3,e_4$ be an arbitrarily orthonormal basis of $\mathfrak{n}$. Then there exists \ $e_5\in \mathfrak{g}\backslash\mathfrak{n}$ whose bracket operations are given by one of the following.\ 
\begin{enumerate}
\item[(1)]
    $ [e_5,e_1]=x_1e_1,\ [e_5,e_2] =x_2e_2,\ [e_5,e_3] =x_3e_3,\ [e_5,e_4]=e_4,\ 0<x_1\le x_2\le x_3\le 1$.
\item[(2)]
  $ [e_5,e_1]=y_1e_1,\ [e_5,e_2] =y_2e_2,\ [e_5,e_3] =e_3,\ [e_5,e_4]=e_3+e_4,\ 0<y_1\le y_2$.
\item[(3)]
 $ [e_5,e_1]=y_1e_1,\ [e_5,e_2] =y_2e_2,\ [e_5,e_3] =e_3+\beta e_4,\ [e_5,e_4]=-\beta e_3+e_4,0<y_1\le y_2,\ 0<\beta$.
\item[(4)]
$ [e_5,e_1]=y_1e_1,\ [e_5,e_2] =e_2,\ [e_5,e_3] =e_2+e_3,\ [e_5,e_4]=e_3+e_4,\ 0<y_1$.
 \item[(5)]
$ [e_5,e_1]=y_1e_1,\ [e_5,e_2] =e_1+y_1e_2,\ [e_5,e_3] =e_3,\ [e_5,e_4]=e_3+e_4,\ 0<y_1$.
\item[(6)]
$ [e_5,e_1]=y_1e_1,\ [e_5,e_2] =e_1+y_1e_2,\ [e_5,e_3] =e_3+\beta e_4,\ [e_5,e_4]=-\beta e_3+e_4,\ 0<y_1,\ 0<\beta$.
\item[(7)]
$ [e_5,e_1]=\alpha e_1+\beta e_2,\ [e_5,e_2] =-\beta e_1+\alpha e_2,\ [e_5,e_3] =e_3+\beta' e_4,\ [e_5,e_4]=-\beta' e_3+e_4,\ 0<\alpha, \beta, \beta'$.
\item[(8)]
$ [e_5,e_1]=e_1,\ [e_5,e_2] =e_1+e_2,\ [e_5,e_3] =e_2+e_3,\ [e_5,e_4]=e_3+e_4$.
\item[(9)]
$ [e_5,e_1]=e_1+\beta e_2,\ [e_5,e_2] =-\beta e_1+e_2,\ [e_5,e_3] =e_1+e_3+\beta' e_4,\ [e_5,e_4]=e_2-\beta' e_3+e_4,\ 0<\beta, \beta'$.
\end{enumerate}
Any two Lie algebras in different types or with different parameters are not isomorphic to each other.
\end{thm}
\subsection{Case of (B)}
$\mathfrak{n}$ has an orthonormal basis $e_1,e_2,e_3,e_4$ with the relations, 
\[
[e_1,e_2]=e_3,\ [e_1,e_3]=e_4,\ \text{others are zero.}
\]
Let $\mathfrak{m}$ be a linear subspace of $\mathfrak{n}$ spanned by $e_1,e_2$, and let $\pi$ be an orthogonal projection from $\mathfrak{n}$ onto $\mathfrak{m}$.
Furthermore, an element of $\mathrm{End}(\mathfrak{n})$ and its representation matrix with respect to the basis $e_1, e_2, e_3, e_4$ are identified.
Under this basis, $\delta^+(\mathfrak{n})$ and $\mathrm{Aut}(\mathfrak{n})$ satisfy the following.
\begin{lem}
\begin{enumerate}
\item[(1)]
The representation matrix of $ D\in \delta^+(\mathfrak{n})$ concerning with $e_1,e_2,e_3,e_4$ takes the following form.
 \begin{align*}
  D=
  \left(
  \begin{array}{cccc}
     x_{11} &x_{12} & 0&0 \\
     x_{21} &x_{22} & 0&0 \\
     x_{31} &x_{32} &x_{11}+x_{22}&0\\
     x_{41} &x_{42} &x_{32}&2x_{11}+x_{22}
  \end{array}
 \right).\tag{4.1}
 \end{align*}
In particular, when
  \[
  D_\mathfrak{m}=\pi\circ D|_{\mathfrak{m}}
 \]
is given, $\mathrm{tr}(D_\mathfrak{m})>0$ holds, and all eigenvalues of $D_\mathfrak{m}$ have positive real parts.
 On the contrary, if the representation matrix of $ D\in \mathrm{End}(\mathfrak{n})$ concerning with $e_1,e_2,e_3,e_4$ satisfies equation (4.1),\ each eigenvalue of $D_\mathfrak{m}=\pi\circ D|_{\mathfrak{m}}$ has a positive real part, and $\mathrm{tr}(D_\mathfrak{m})>0$,  then $D\in \delta^+(\mathfrak{n})$.
 \item[(2)]The representation matrix of $ A\in \mathrm{Aut}(\mathfrak{n})$ concerning with $e_1,e_2,e_3,e_4$ takes the following form.
 \begin{align*}
 A=
 \left(
  \begin{array}{cccc}
     y_{11} &y_{12} & 0&0 \\
     y_{21} &y_{22} & 0&0 \\
     y_{31} &y_{32} &y_{11}y_{22}-y_{12}y_{21}&0\\
     y_{41} &y_{42} &y_{11}y_{32}-y_{12}y_{31}&y_{11}(y_{11}y_{22}-y_{12}y_{21})
  \end{array}
  \right).\tag{4.2}
 \end{align*}
 In particular, when
  \[
  A_\mathfrak{m}=\pi\circ A|_{\mathfrak{m}},
 \]
$\det(A_\mathfrak{m})\neq 0$ is satisfied.\\ 
 On the contrary, if the representation matrix of $ A\in \mathrm{GL}(\mathfrak{n})  $concerning with $e_1,e_2,e_3.e_4$ satisfies equation (4.2), then $A\in \mathrm{Aut}(\mathfrak{n})$.
 \end{enumerate}
  \end{lem}
\begin{proof}
We denote that
  \[
  D(e_1)=x_{11}e_1+x_{21}e_2+x_{31}e_3+x_{41}e_4,\ D(e_2)=x_{12}e_1+x_{22}e_2+x_{32}e_3+x_{42}e_4,\]
\[
D(e_3)=x_{13}e_1+x_{23}e_2+x_{33}e_3+x_{43}e_4,
  \]
with respect to $D\in \delta^+(\mathfrak{n})$. The following holds from $D\in \delta^+(\mathfrak{n})$.
    \begin{align*}
 D([e_1,e_2])
 &=[D(e_1),e_2]+[e_1,D(e_2)]\\
 &=x_{11}e_3+x_{22}e_3+x_{32}e_4\\
 &=(x_{11}+x_{22})e_3+x_{32}e_4.
 \end{align*}
Transforming both sides, 
\[D(e_3)=\mathrm{tr}(D_\mathfrak{m})e_3+x_{32}e_4.\]
Therefore, $x_{13},x_{23}=0,\ x_{33}=\mathrm{tr}(D_\mathfrak{m})$ and $x_{43}=x_{32}$. Moreover,
  \begin{align*}
 D([e_1,e_3])
 &=[D(e_1),e_3]+[e_1,D(e_3)]\\
 &=x_{11}e_3+x_{23}e_3+x_{33}e_4\\
 &=(2x_{11}+x_{22})e_4. 
 \end{align*}
Therefore, the representation matrix satisfies (4.1).\\
 Conversely, when the representation matrix of $D\in \mathrm{End}(\mathfrak{n})$ has the form of (4.1), $D$ follows 
 \[
  D([e_1,e_2])=\mathrm{tr}(D_\mathfrak{m})e_3+x_{32}e_4=[D(e_1),e_2]+[e_1,D(e_2)],
  \]
 \[
  D([e_1,e_3])=(2x_{11}+x_{22})e_4=[D(e_1),e_3]+[e_1,D(e_3)].
  \]
Therefore, $D\in \delta(\mathfrak{n})$. Furthermore, from the assumptions, since all eigenvalues of $D_\mathfrak{m}$ have positive real parts, $D\in \delta^+(\mathfrak{n})$ follows.
 Next, we denote that
 \[
  A(e_1)=y_{11}e_1+y_{21}e_2+y_{31}e_3+y_{41}e_4\ ,\ A(e_2)=y_{12}e_1+y_{22}e_2+y_{32}e_3+y_{42}e_4,
  \]
\[
  A(e_3)=y_{13}e_1+y_{23}e_2+y_{33}e_3+y_{43}e_4,
  \]
with respect to $A\in \mathrm{Aut}(\mathfrak{n})$. The following holds from $A\in \mathrm{Aut}(\mathfrak{n})$.
  \begin{align*}
 A([e_1,e_2])
 &=[A(e_1),A(e_2)]\\
 &=[y_{11}e_1+y_{21}e_2+y_{31}e_3+y_{41}e_4,y_{12}e_1+y_{22}e_2+y_{32}e_3+y_{42}e_4]\\
 &=[y_{11}e_1,y_{22}e_2+y_{32}e_3]+[y_{21}e_2+y_{31}e_3,y_{12}e_1]\\
 &=y_{11}y_{22}e_3-y_{12}y_{21}e_3+y_{11}y_{32}e_4-y_{12}y_{31}e_4\\
 &=\det(A_\mathfrak{m})e_3+(y_{11}y_{32}-y_{12}y_{31})e_4.
 \end{align*}
Transforming both sides, 
\[A(e_3)=\det(A_\mathfrak{m})e_3+(y_{11}y_{32}-y_{12}y_{31})e_4.\]
Therefore, $y_{13},\ y_{23}=0,\ y_{33}=\det(A_\mathfrak{m}),\ x_{43}=y_{11}y_{32}-y_{12}y_{31}$. Moreover,
  \begin{align*}
 A([e_1,e_3])
 &=[A(e_1),A(e_3)]\\
 &=[y_{11}e_1+y_{21}e_2+y_{31}e_3+y_{41}e_4,y_{13}e_1+y_{23}e_2+y_{33}e_3+y_{43}e_4]\\
 &=[y_{11}e_1,y_{23}e_2+y_{33}e_3]+[y_{21}e_2+y_{31}e_3,y_{13}e_1]\\
 &=y_{11}y_{23}e_3-y_{21}y_{13}e_3+y_{11}y_{33}e_4-y_{13}y_{31}e_4\\
 &=y_{11}\det(A_\mathfrak{m})e_4.
 \end{align*}
Transforming both sides, 
\[
A(e_4)=y_{11}\det(A_\mathfrak{m})e_4.
\] 
Therefore, the representation matrix satisfies (4.2).
 Conversely, when the representation matrix of  $A\in \mathrm{GL}(\mathfrak{n})$ has the form of (4.2), the following holds. 
  \[
  A([e_1,e_2])=\det(A_\mathfrak{m})e_3+(y_{11}y_{32}-y_{12}y_{31})e_4=[A(e_1),A(e_2)],
  \]
\[
 A([e_1,e_3])=y_{11}\det(A_\mathfrak{m})e_4=[A(e_1),A(e_3)].
 \]
Furthermore, from the assumptions, since $\det(A)\neq 0$, $A\in \mathrm{Aut}(\mathfrak{n})$ follows.
\end{proof}
 Furthermore, the following theorem holds.
  \begin{lem}\label{lem:b1}
For any $D\in \delta^+(\mathfrak{n})$, there exists unique $\Tilde{D}\in \delta^+(\mathfrak{n})$, that is, conjugate to $D$ by $\mathrm{Aut}(\mathfrak{n})$ such that the representation matrix of $\Tilde{D}$ concerning with $e_1,e_2,e_3,e_4$ takes the following form.
\begin{align}
  \Tilde{D}=
  \left(
  \begin{array}{cccc}
     x_{11} &x_{12} & 0 &0\\
     x_{21} &x_{22} & 0 &0\\
     0 &0 &x_{11}+x_{22} &0\\
     y_1 &y_2 &0 &2x_{11}+x_{22}
  \end{array}
 \right).\tag{4.1}
 \end{align}
  Here,\ 
  \begin{align*}
 y_1&=x_{22}(x_{31}x_{32}+x_{11}x_{41})-x_{21}(x_{22}x_{32}+x_{11}x_{42}),\\
 y_2&=x_{11}(x_{32}x_{32}+x_{11}x_{42})-x_{12}(x_{31}x_{32}+x_{11}x_{41}).
 \end{align*}
  \end{lem}
  \begin{proof}
We take an element of $\mathrm{Aut}(\mathfrak{n})$ whose representation matrix concerning with $e_1,e_2,e_3,e_4$ is
  \[A=
  \left(
  \begin{array}{cccc}
     x_{11} &x_{12} & 0 &0\\
     x_{21} &x_{22} & 0 &0\\
     x_{11}x_{31}+x_{21}x_{32} &x_{22}x_{32}+x_{12}x_{31} &\det\Tilde{X}&0\\
     0 &0 &x_{32}\det\Tilde{X} &x_{11}\det\Tilde{X}
  \end{array}
 \right),
  \]
where
\[
   \Tilde{X}=
  \left(
  \begin{array}{cc}
     x_{11} &x_{12} \\
     x_{21} &x_{22} 
  \end{array}
 \right).
  \]
Since it becomes $AD=\Tilde{D}A$, $D$ and $\Tilde{D}$ are conjugate in $\mathrm{Aut}(\mathfrak{n})$. 
  \end{proof}
From Lemma \ref{lem:b1}, we may assume that the representation matrix of any $D\in \delta^+(\mathfrak{n})$ is expressed in the form of (4.1). Under the assumption, the following holds.
  \begin{lem}\label{lem:b4.5}
For any $D\in \delta^+(\mathfrak{n})$, there exists unique $\Tilde{D}\in \delta^+(\mathfrak{n})$, that is, conjugate to $D$ by $\mathrm{Aut}(\mathfrak{n})$ such that the representation matrix of $\Tilde{D}$ concerning with $e_1,e_2,e_3,e_4$ takes the following form.
 \[\Tilde{D}=
  \left(
  \begin{array}{cccc}
     x_{11} &x_{12} & 0 &0\\
     x_{21} &x_{22} & 0 &0\\
     0 &0 &x_{11}+x_{22} &0\\
     0 &0 &0 &2x_{11}+x_{22}
  \end{array}
 \right).
  \]
\end{lem}
\begin{proof}
First, for any $D\in \delta^+(\mathfrak{n})$, we show that $x_{11}(2x_{11}+x_{22})+\det\Tilde{X}\neq0$, where
\[
   \Tilde{X}=
  \left(
  \begin{array}{cc}
     x_{11} &x_{12} \\
     x_{21} &x_{22} 
  \end{array}
 \right).\]
Assume that $x_{11}(2x_{11}+x_{22})+\det\Tilde{X}=0$. Then, eigenvalues $\lambda_1,\ \lambda_2$ of $\Tilde{X}$ are expressed as follows.
\begin{align*}
\lambda_1&=\frac{(x_{11}+x_{22})+\sqrt{(x_{11}+x_{22})^2-4x_{11}(2x_{11}+x_{22})}}{2},\\
\lambda_2&=\frac{(x_{11}+x_{22})-\sqrt{(x_{11}+x_{22})^2-4x_{11}(2x_{11}+x_{22})}}{2}.
\end{align*}
From $D\in \delta^+(\mathfrak{n})$, we have $x_{11}+x_{22}>0$. Since $\lambda_2$ has a positive real part, the following holds.
\[
x_{11}(2x_{11}+x_{22})>0.
\]
On the other hand, since $D\in \delta^+(\mathfrak{n})$, we have $\det\Tilde{X}>0$, and from the assumption,
\[
x_{11}(2x_{11}+x_{22})=-\det\Tilde{X},
\]
which implies
\[
x_{11}(2x_{11}+x_{22})<0,
\]
so this is a contradiction.
Thus, for any $D\in \delta^+(\mathfrak{n})$, $x_{11}(2x_{11}+x_{22})+\det\Tilde{X}\neq0$.
Second, we take an element of $\mathrm{Aut}(\mathfrak{n})$ whose representation matrix concerning with $e_1,e_2,e_3,e_4$ is
  \[A=
  \left(
  \begin{array}{cccc}
     x_{11} &x_{12} & 0 &0\\
     x_{21} &x_{22} & 0 &0\\
     0 &0 &\det\Tilde{X}&0\\
     a &b &x_{32}\det\Tilde{X} &x_{11}\det\Tilde{X}
  \end{array}
 \right).
  \]
Here, $a,b\in \mathbb{R}$ are
  \begin{align*}
a&=\frac{(2y_1x_{11}+y_2x_{21})(2x_{11}+x_{22})}{x_{11}(2x_{11}+x_{22})+\det{\Tilde{X}}},\\
b&=\frac{(y_1x_{12}+y_2x_{11}+y_2x_{22})(2x_{11}+x_{22})}{x_{11}(2x_{11}+x_{22})+\det{\Tilde{X}}}.
\end{align*}
Since it becomes $AD=\Tilde{D}A$, $D$ and $\Tilde{D}$ are conjugate in $\mathrm{Aut}(\mathfrak{n})$. 
\end{proof}
From the above discussion, Case of (B) is reduced to the case of four-dimensional SNC-algebras with non-abelian derived algebra by Lemma \ref{lem:b4.5}. Hence, the following holds by Lemma \ref{lem:3b} and Proposition \ref{prop:3b}.
\begin{prop}
The representation matrix of $D\in\delta^+(\mathfrak{n})$ concerning with $e_1,e_2,e_3,e_4$ is conjugate to one of the following.
\[
  \left(
  \begin{array}{cccc}
     1 &0 & 0 &0\\
     0 &x & 0 &0\\
     0 &0 &1+x &0\\
     0 &0 &0 &2+x
  \end{array}
 \right),
 \left(
  \begin{array}{cccc}
     1 &1 & 0 &0\\
     0 &1 & 0 &0\\
     0 &0 &2 &0\\
     0 &0 &0 &3
  \end{array}
 \right),
 \left(
  \begin{array}{cccc}
     1 & -\beta & 0 & 0\\
     \beta &1 & 0 & 0\\
     0 &0 &2 &0\\
     0 &0 &0 &3
  \end{array}
 \right),
\]
where, $0<x,\beta$.\ 
\end{prop}
\begin{thm}
Let $\mathfrak{g}$ be an SNC-algebra of dimension five with non-abelian derived algebra $\mathfrak{n}=[\mathfrak{g},\mathfrak{g}]$, which has an orthonormal basis $e_1,e_2,e_3,e_4$ with the relations
\[
[e_1,e_2]=e_3,\ [e_1,e_3]=e_4,\ \text{others are zero.}
\]
Then there exists $e_5\in \mathfrak{g}\backslash\mathfrak{n}$ whose bracket operations are given by one of the following.\ 
\begin{enumerate}
\item[(1)]
$ [e_5,e_1]=e_1,\ [e_5,e_2] =xe_2,\ [e_5,e_3] =(1+x)e_3,\ [e_5,e_4]=(2+x)e_4,\ 0<x.$
\item[(2)]
$ [e_5,e_1]=e_1,\ [e_5,e_2] =e_1+e_2,\ [e_5,e_3] =2e_3,\ [e_5,e_4]=3e_4.$
\item[(3)]
 $ [e_5,e_1]=e_1+\beta e_2,\ [e_5,e_2] =-\beta e_1+e_2,\ [e_5,e_3] =2e_3,\ [e_5,e_4]=3e_4,0<\beta.$
\end{enumerate}
Any two Lie algebras in different types or with different parameters are not isomorphic to each other.
\end{thm}
\subsection{Case of (C)}
$\mathfrak{n}$ has an orthonormal basis $e_1,e_2,e_3,e_4$ with the relations, 
\[
[e_1,e_2]=e_4,\ \text{others are zero.}
\]

Let $\mathfrak{m}$ be a linear subspace of $\mathfrak{n}$ spanned by $e_1,e_2,e_3$ and $\mathfrak{h}$ be a linear subspace of $\mathfrak{n}$ spanned by $e_1,e_2$.
Let $\pi_1$ be an orthogonal projection from $\mathfrak{n}$ onto $\mathfrak{m}$ and $\pi_2$ be an orthogonal projection from $\mathfrak{n}$ onto $\mathfrak{h}$.
Furthermore, an element of $\mathrm{End}(\mathfrak{n})$ and its representation matrix with respect to the basis $e_1, e_2, e_3, e_4$ are identified.
Under this basis, $\delta^+(\mathfrak{n})$ and $\mathrm{Aut}(\mathfrak{n})$ satisfy the following.
\begin{lem}\label{lem:5c}
\begin{enumerate}
\item[(1)]
The representation matrix of $ D\in \delta^+(\mathfrak{n})$ concerning with $e_1,e_2,e_3,e_4$ takes the following form.
 \begin{align}
  D=
  \left(
  \begin{array}{cccc}
     x_{11} &x_{12} &0&0 \\
     x_{21} &x_{22} &0&0 \\
     x_{31} &x_{32} &x_{33}&0 \\
     x_{41} &x_{42} &x_{43}&x_{11}+x_{22}
  \end{array}
 \right).\tag{4.3}
 \end{align}
  In particular, $x_{11}+x_{22}>0$ holds, and all eigenvalues of $D_\mathfrak{m}:=\pi_1\circ D|_{\mathfrak{m}}$ have positive real parts.
 On the contrary, if the representation matrix of $ D\in \mathrm{End}(\mathfrak{n})$ concerning with $e_1,e_2,e_3,e_4$ satisfies equation (4.3), as well as $x_{11}+x_{22}>0$ and each eigenvalue of $D_\mathfrak{m}$ has a positive real part,  then $D\in \delta^+(\mathfrak{n})$.
 \item[(2)]The representation matrix of $ A\in \mathrm{Aut}(\mathfrak{n})$ concerning with $e_1,e_2,e_3,e_4$ takes the following form:
 \begin{align*}
 A=
 \left(
  \begin{array}{cccc}
     y_{11} &y_{12} & 0&0 \\
     y_{21} &y_{22} & 0&0 \\
     y_{31} &y_{32} & y_{33}&0\\
     y_{41} &y_{42} & y_{43}&y_{11}y_{22}-y_{12}y_{21}
  \end{array}
  \right).\tag{4.4}
 \end{align*}
In particular, when we set 
  \[
  A_\mathfrak{m}:=\pi_1\circ A|_{\mathfrak{m}},
 \]
$\det(A_\mathfrak{m})\neq 0$ holds.
On the contrary, if the representation matrix of $ A\in \mathrm{GL}(\mathfrak{n})  $concerning with $e_1,e_2,e_3,e_4$ satisfies equation (4.4), then $A\in \mathrm{Aut}(\mathfrak{n})$. 
 \end{enumerate}
  \end{lem}
\begin{proof}
 We denote that
  \[
  D(e_1)=x_{11}e_1+x_{21}e_2+x_{31}e_3+x_{41}e_4,\ D(e_2)=x_{12}e_1+x_{22}e_2+x_{32}e_3+x_{42}e_4,\]
\[
D(e_3)=x_{13}e_1+x_{23}e_2+x_{33}e_3+x_{43}e_4,
  \]
with respect to $D\in \delta^+(\mathfrak{n})$. The following holds from $D\in \delta^+(\mathfrak{n})$.
\[
 D([e_1,e_2])=[D(e_1),e_2]+[e_1,D(e_2)]=x_{11}e_4+x_{22}e_4=(x_{11}+x_{22})e_4.
 \]
Hence, $D(e_3)=\mathrm{tr}(D_\mathfrak{h})e_4$ where $D_\mathrm{h}:=\pi_2\circ D|_{\mathfrak{h}}$. Therefore, $x_{14},x_{24},x_{34}=0,\  x_{44}=\mathrm{tr}(D_\mathfrak{h})$. Moreover,
\[
 D([e_1,e_3])=[D(e_1),e_3]+[e_1,D(e_3)]=x_{23}e_4.
\]
By $D([e_1,e_3])=0$, $x_{23}=0$. Similarly, by $D([e_2,e_3])=0$, $x_{13}=0$. Therefore, the representation matrix satisfies (4.3).
  Conversely, when the representation matrix of $D\in \mathrm{End}(\mathfrak{n})$ has the form of (4.3), $D$ satisfies
 \[
  D([e_1,e_2])=\mathrm{tr}(D_\mathfrak{h})e_4=[D(e_1),e_2]+[e_1,D(e_2)].
  \]
  Therefore,\ $D\in \delta(\mathfrak{n})$.\ Furthermore, since all eigenvalues of $D_\mathfrak{m}$ have positive real parts by the assumptions, $D\in \delta^+(\mathfrak{n})$ follows.
 
Next, we denote that
 \[
  A(e_1)=y_{11}e_1+y_{21}e_2+y_{31}e_3+y_{41}e_4,\ A(e_2)=y_{12}e_1+y_{22}e_2+y_{32}e_3+y_{42}e_4,
  \]
\[
  A(e_3)=y_{13}e_1+y_{23}e_2+y_{33}e_3+y_{43}e_4,
  \]
with respect to $A\in \mathrm{Aut}(\mathfrak{n})$. The following holds.
  \begin{align*}
 A([e_1,e_2])
 &=[A(e_1),A(e_2)]\\
 &=[y_{11}e_1+y_{21}e_2+y_{31}e_3+y_{41}e_4,y_{12}e_1+y_{22}e_2+y_{32}e_3+y_{42}e_4]\\
 &=[y_{11}e_1,y_{22}e_2]+[y_{21}e_2,y_{12}e_1]\\
 &=y_{11}y_{22}e_4-y_{12}y_{21}e_4\\
 &=\det(A_\mathfrak{h})e_4,
 \end{align*}
where $A_\mathrm{h}:=\pi_2\circ A|_{\mathfrak{h}}$. Hence, $A(e_4)=\det(A_\mathfrak{h})e_4$. Therefore, $y_{14},y_{24},y_{34}=0$ and $y_{44}=\det(A_\mathfrak{h})$. Moreover,
  \begin{align*}
 A([e_1,e_3])
 &=[A(e_1),A(e_3)]\\
 &=[y_{11}e_1+y_{21}e_2+y_{31}e_3+y_{41}e_4,y_{13}e_1+y_{23}e_2+y_{33}e_3+y_{43}e_4]\\
 &=[y_{11}e_1,y_{23}e_2]+[y_{21}e_2,y_{13}e_1]\\
 &=(y_{11}y_{23}-y_{21}y_{13})e_4,\\
 A([e_2,e_3])
 &=[A(e_2),A(e_3)]\\
 &=[y_{12}e_1+y_{22}e_2+y_{32}e_3+y_{42}e_4,y_{13}e_1+y_{23}e_2+y_{33}e_3+y_{43}e_4]\\
 &=[y_{12}e_1,y_{23}e_2]+[y_{22}e_2,y_{13}e_1]\\
 &=(y_{12}y_{23}-y_{22}y_{13})e_4.
 \end{align*}
By $[e_1,e_3]=[e_2,e_3]=0$ and $\det(A_\mathfrak{h})\neq 0$, $y_{13},y_{23}=0$. Therefore, the representation matrix $A$ satisfies (4.4).
  Conversely, when the representation matrix of $A\in \mathrm{GL}(\mathfrak{n})$ has the form of (4.4), $A$ satisfies
  \[
  A([e_1,e_2])=\det(A_\mathfrak{h})e_4=[A(e_1),A(e_2)].
  \]
 Therefore, from the assumptions and $\det(A)\neq 0$, $A\in \mathrm{Aut}(\mathfrak{n})$ holds.
\end{proof}
By Lemma \ref{lem:5c}, the following holds.
\begin{lem}\label{lem:c2}
 For any $D\in \delta^+(\mathfrak{n})$ which is represented in the form (4.3), there exists unique $\Tilde{D}\in \delta^+(\mathfrak{n})$ that is conjugate to $D$ by $\mathrm{Aut}(\mathfrak{n})$ such that the representation matrix of $\Tilde{D}$ concerning with $e_1,e_2,e_3,e_4$ takes the following form:
  \begin{align*}
 \Tilde{D}=
 \left(
  \begin{array}{c|c}
    X &O \\ \hline
   X_1&X_2  
   \end{array}
\right).
 \end{align*}
Here, $X,X_2\in\mathrm{GL}(2,\mathbb{R}),X_1\in\mathrm{M}(2,\mathbb{R})$, and $X$ is expressed in one of the following forms.
\begin{align}
 \left(
  \begin{array}{cc}
    \lambda_1 &0 \\
   0&\lambda_2  
   \end{array}
\right),\ 
\left(
  \begin{array}{cc}
    \lambda_1 &1 \\
   0&\lambda_1  
   \end{array}
\right),\ 
\left(
  \begin{array}{cc}
    \lambda_1 &-\alpha \\
   \alpha&\lambda_1  
   \end{array}
\right),
 \end{align}
where s $0<\lambda_1,\lambda_2,\alpha$.
\end{lem}
\begin{proof}
We set
\begin{align*}
\Tilde{X}=
 \left(
  \begin{array}{cc}
    x_{11} &x_{12} \\
   x_{21}&x_{22}  
   \end{array}
\right).
\end{align*}
We have $A_1\in\mathrm{GL}(2,\mathbb{R})$ such that $\Tilde{X}$ is conjugate to one of the matrices in (1) by $A_1$. We define $A$ as

 \begin{align*}
A=
 \left(
  \begin{array}{cc|cc}
    A_1 && \\ \hline
   &&1&0\\
&&0&\det(A_1)  
   \end{array}
\right).
 \end{align*}
Since it becomes $AD=\Tilde{D}A$, $D$ and $\Tilde{D}$ are conjugate in $\mathrm{Aut}(\mathfrak{n})$. 
\end{proof}
From Lemma \ref{lem:c2}, the representation matrix of $D_{\mathfrak{h}}$ can be treated as (1).
Regarding the SNC-algebra of (C), it is classified as follows depending on the value of $x_{33}$.
   {\bf{ (C-1)}} Case of $x_{33}\neq x_{11}+x_{22}$ or $x_{43}=0$
  \begin{lem}
For any $D\in \delta^+(\mathfrak{n})$ which satisfies (C-1), there exists unique $D_1\in \delta^+(\mathfrak{n})$, that is, conjugate to $D$ by $\mathrm{Aut}(\mathfrak{n})$ such that the representation matrix of $D_1$ concerning with $e_1,e_2,e_3,e_4$ takes the following form.
  \begin{align}
  D_1=
  \left(
  \begin{array}{cccc}
     x_{11} &x_{12} &0&0 \\
     x_{21} &x_{22} &0&0 \\
     x_{31} &x_{32} &x_{33}&0 \\
     0&0&0&x_{11}+x_{22}
  \end{array}
 \right).\tag{4.5}
 \end{align}
\end{lem}
  \begin{proof}
Case of $x_{33}\neq x_{11}+x_{22}$.
We take an element of $\mathrm{Aut}(\mathfrak{n})$ whose representation matrix concerning with $e_1,e_2,e_3,e_4$ is
  \[A=
  \left(
  \begin{array}{cccc}
     x_{11} &x_{12} & 0 &0\\
     x_{21} &x_{22} & 0 &0\\
     x_{31} &x_{32} & x_{33}&0\\
     y_1 &y_2 &y_3 &x_{11}x_{22}-x_{12}x_{21}
  \end{array}
\right).
  \]
Here, $y_1,y_2,y_3 \in \mathbb{R}$ are
  \begin{align*}
 y_1&=x_{11}\left(x_{41}+\frac{x_{31}x_{43}}{x_{11}+x_{22}-x_{33}}\right)+x_{21}\left(x_{42}+\frac{x_{32}x_{43}}{x_{11}+x_{22}-x_{33}}\right),\\
 y_2&=x_{12}\left(x_{41}+\frac{x_{31}x_{43}}{x_{11}+x_{22}-x_{33}}\right)+x_{22}\left(x_{42}+\frac{x_{32}x_{43}}{x_{11}+x_{22}-x_{33}}\right),\\
 y_3&=\frac{x_{43}(x_{11}x_{22}-x_{12}x_{21})}{x_{11}+x_{22}-x_{33}}.
 \end{align*}
  Since it becomes $AD= D_1A$, $D$ and $D_1$ are conjugate in $\mathrm{Aut}(\mathfrak{n})$. 

Case of $x_{43}=0$.
We take an element of $\mathrm{Aut}(\mathfrak{n})$ whose representation matrix concerning with $e_1,e_2,e_3,e_4$ is
  \[A=
  \left(
  \begin{array}{cccc}
     x_{11} &x_{12} & 0 &0\\
     x_{21} &x_{22} & 0 &0\\
     x_{31} &x_{32} & x_{33}&0\\
     y_1 &y_2 &0&x_{11}x_{22}-x_{12}x_{21}
   \end{array}
\right).
  \]
Here, $y_1,y_2 \in \mathbb{R}$ are
  \begin{align*}
 y_1&=x_{11}x_{41}+x_{21}x_{42},\\
 y_2&=x_{12}x_{41}+x_{22}x_{42}.
 \end{align*}
 Since it becomes $AD= D_1A$, $D$ and $D_1$ are conjugate in $\mathrm{Aut}(\mathfrak{n})$. 
  \end{proof}
Let $\Tilde{X}$ in $\mathrm{GL}(\mathfrak{h})$ as follows.
\[
   \Tilde{X}=
  \left(
  \begin{array}{cc}
     x_{11} &x_{12} \\
     x_{21} &x_{22} 
  \end{array}
 \right).
  \]
When real parts of eigenvalues of $\Tilde{X}$ are denoted as $\lambda_1,\lambda_2$, they are classified into following cases.
 
{\bf{ (C-1-1)}}Case of $x_{31},x_{32}=0$ or $x_{33}\neq \lambda_1,\lambda_2$ or $\lambda_1\neq\lambda_2, x_{33}=\lambda_1,x_{31}=0$ or $\lambda_1\neq\lambda_2, x_{33}=\lambda_2,x_{32}=0$.
\begin{lem}
For any $D\in \delta^+(\mathfrak{n})$ which satisfies (4.5), there exists unique $D_{11}\in \delta^+(\mathfrak{n})$, that is, conjugate to $D$ by $\mathrm{Aut}(\mathfrak{n})$ such that the representation matrix of $D_{11}$ concerning with $e_1,e_2,e_3,e_4$ takes the following form.
  \begin{align}
  D_{11}=
  \left(
  \begin{array}{cccc}
     x_{11} &x_{12} &0&0 \\
     x_{21} &x_{22} &0&0 \\
     0 &0&x_{33}&0 \\
     0&0&0&x_{11}+x_{22}
  \end{array}
 \right).\tag{4.6}
 \end{align} 
\end{lem}
  \begin{proof}
The case where $x_{31}=x_{32}=0$ is trivial. First, we show the case of $x_{33}\neq \lambda_1,\lambda_2$.
We take an element of $\mathrm{Aut}(\mathfrak{n})$ whose representation matrix concerning with $e_1,e_2,e_3,e_4$ is
  \[A=
  \left(
  \begin{array}{cccc}
     x_{11} &x_{12} & 0 &0\\
     x_{21} &x_{22} & 0 &0\\
     y_1 &y_2 & x_{33}&0\\
     0 &0 &0&\det\Tilde{X}
  \end{array}
\right).
  \]
Here, $y_1,y_2 \in \mathbb{R}$ are
  \begin{align*}
 y_1&=\frac{(x_{11}x_{31}+x_{21}x_{32})(x_{22}-x_{33})-x_{21}(x_{12}x_{31}+x_{22}x_{32})}{(x_{33}-\lambda_1)(x_{33}-\lambda_2)},\\
 y_2&=\frac{(x_{12}x_{31}+x_{22}x_{32})(x_{11}-x_{33})-x_{12}(x_{11}x_{31}+x_{21}x_{32})}{(x_{33}-\lambda_1)(x_{33}-\lambda_2)}.
 \end{align*}
Since it becomes $AD= D_{11}A$, $D$, and $D_{11}$ are conjugate in $\mathrm{Aut}(\mathfrak{n})$. 
Next, we show the case of $\lambda_1\neq\lambda_2, x_{33}=\lambda_1,x_{31}=0$.
The representation matrices of $D$ and $D_{11}$ satisfy the following.
 \begin{align*}
  D=
  \left(
  \begin{array}{cccc}
     \lambda_1 &0 &0&0 \\
     0&\lambda_2 &0&0 \\
     0&x_{32} &\lambda_1&0 \\
     0&0&0&\lambda_1+\lambda_2
  \end{array}
 \right),
 \end{align*}
\begin{align*}
  D_{11}=
  \left(
  \begin{array}{cccc}
  \lambda_1 &0 &0&0 \\
     0&\lambda_2 &0&0 \\
     0&0&\lambda_1&0 \\
     0&0&0&\lambda_1+\lambda_2
  \end{array}
 \right).
 \end{align*} 
We take an element of $\mathrm{Aut}(\mathfrak{n})$ whose representation matrix concerning with $e_1,e_2,e_3,e_4$ is
  \[A=
  \left(
  \begin{array}{cccc}
    \lambda_1 &0 &0&0 \\
     0&\lambda_2 &0&0 \\
     0&\frac{\lambda_1x_{32}}{\lambda_1-\lambda_2}&\lambda_1&0 \\
     0 &0 &0&\lambda_1\lambda_2
  \end{array}
\right).
  \]
Since it becomes $AD= D_{11}A$, $D$ and $D_{11}$ are conjugate in $\mathrm{Aut}(\mathfrak{n})$. 
Last, we show the case of $\lambda_1\neq\lambda_2, x_{33}=\lambda_2,x_{32}=0$.
The representation matrices of $D$ and $D_{11}$ satisfy the following.
 \begin{align*}
  D=
  \left(
  \begin{array}{cccc}
     \lambda_1 &0 &0&0 \\
     0&\lambda_2 &0&0 \\
     x_{31}&0 &\lambda_2&0 \\
     0&0&0&\lambda_1+\lambda_2
  \end{array}
 \right),
 \end{align*}
\begin{align*}
  D_{11}=
  \left(
  \begin{array}{cccc}
  \lambda_1 &0 &0&0 \\
     0&\lambda_2 &0&0 \\
     0&0&\lambda_2&0 \\
     0&0&0&\lambda_1+\lambda_2
  \end{array}
 \right).
 \end{align*} 
We take an element of $\mathrm{Aut}(\mathfrak{n})$ whose representation matrix concerning with $e_1,e_2,e_3,e_4$ is
  \[A=
  \left(
  \begin{array}{cccc}
    \lambda_1 &0 &0&0 \\
     0&\lambda_2 &0&0 \\
     \frac{\lambda_1x_{31}}{\lambda_2-\lambda_1}&0&\lambda_1&0 \\
     0 &0 &0&\lambda_1\lambda_2
  \end{array}
\right).
  \]
Since it becomes $AD= D_{11}A$, $D$ and $D_{11}$ are conjugate in $\mathrm{Aut}(\mathfrak{n})$. 
  \end{proof}
When $D_1\in \delta^+(\mathfrak{n})$ is represented in the form of (4.6), it can be reduced to the classification of the abelian case in the four-dimensional SNC-algebra, so the following holds.
\begin{prop}
The representation matrix of $D\in\delta^+(\mathfrak{n})$ which satisfies (C-1-1), concerning with $e_1,e_2,e_3,e_4$ is conjugate to one of the following.
\begin{align*}
  \left(
  \begin{array}{cccc}
     1 &0&0&0 \\
     0 &x_1 &0&0 \\
     0&0&x_2&0 \\
     0&0&0&1+x_1
  \end{array}
 \right),\ 
 \left(
  \begin{array}{cccc}
     1 &1&0&0 \\
     0 &1&0&0 \\
     0&0&x&0 \\
     0&0&0&2
  \end{array}
 \right),\ 
 \left(
  \begin{array}{cccc}
     1 &-\alpha&0&0 \\
     \alpha &1&0&0 \\
     0&0&x&0 \\
     0&0&0&2
  \end{array}
 \right),
 \end{align*}
 where,\ $0<x_1\leq 1, 0<\alpha,x,x_2$.
\end{prop}
 {\bf{ (C-1-2)}}Case of the representation matrix of $D$ is given as follows.
  \begin{align*}
  D=
  \left(
  \begin{array}{cccc}
     \lambda_1 &0 &0&0 \\
     0&\lambda_2&0&0 \\
     x_{31} &0&\lambda_1&0 \\
     0&0&0&\lambda_1+\lambda_2
  \end{array}
 \right),\tag{4.7}
 \end{align*}
where, $x_{31}\neq0, \lambda_1\leq \lambda_2$.
 \begin{lem}
For any $D\in \delta^+(\mathfrak{n})$ which satisfies (4.7), there exists unique $D_{12}\in \delta^+(\mathfrak{n})$, that is, conjugate to $D$ by $\mathrm{Aut}(\mathfrak{n})$ such that the representation matrix of $D_{12}$ concerning with $e_1,e_2,e_3,e_4$ takes the following form.
   \begin{align*}
  D_{12}=
  \left(
  \begin{array}{cccc}
     \lambda_1 &0 &0&0 \\
     0&\lambda_2&0&0 \\
     1 &0&\lambda_1&0 \\
     0&0&0&\lambda_1+\lambda_2
  \end{array}
 \right).
 \end{align*}
\end{lem}
\begin{proof}
 We take an element of $\mathrm{Aut}(\mathfrak{n})$ whose representation matrix concerning with $e_1,e_2,e_3,e_4$ is
  \[A=
   \left(
  \begin{array}{cccc}
     \lambda_1 &0 &0&0 \\
     0&\lambda_2&0&0 \\
     0 &0&\frac{\lambda_1}{x_{31}}&0 \\
     0&0&0&\lambda_1\lambda_2
  \end{array}
 \right).
  \]
Since it becomes $AD= D_{12}A$, $D$ and $D_{12}$ are conjugate in $\mathrm{Aut}(\mathfrak{n})$. 
\end{proof}
From the above discussion, the following is satisfied.
\begin{prop}
The representation matrix of $D\in\delta^+(\mathfrak{n})$ which satisfies (C-1-2), concerning with $e_1,e_2,e_3,e_4$ is conjugate with the following.
\begin{align*}
  \left(
  \begin{array}{cccc}
     1 &0 &0&0 \\
     0&x&0&0 \\
     1 &0&1&0 \\
     0&0&0&1+x
  \end{array}
 \right),
 \end{align*}
 where,\ $1\leq x$.
\end{prop}
{\bf{ (C-1-3)}}Case of the representation matrix of $D$ is given as follows.
  \begin{align*}
  \left(
  \begin{array}{cccc}
     \alpha&-\beta &0&0 \\
     \beta&\alpha&0&0 \\
     x_{31} &0&\alpha&0 \\
     0&0&0&2\alpha
  \end{array}
 \right),\tag{4.8}
 \end{align*}
where, $0<\beta$.
 \begin{lem}
For any $D\in \delta^+(\mathfrak{n})$ which satisfies (4.8), there exists unique $D_{12}\in \delta^+(\mathfrak{n})$, that is, conjugate to $D$ by $\mathrm{Aut}(\mathfrak{n})$ such that the iepresentation matrix of $D_{12}$ concerning with $e_1,e_2,e_3,e_4$ takes the following form.
   \begin{align*}
  D_{13}=
  \left(
  \begin{array}{cccc}
     \alpha&-\beta &0&0 \\
     \beta&\alpha&0&0 \\
     1 &0&\alpha&0 \\
     0&0&0&2\alpha
  \end{array}
 \right).
 \end{align*}
\end{lem}
\begin{proof}
 We take an element of $\mathrm{Aut}(\mathfrak{n})$ whose representation matrix concerning with $e_1,e_2,e_3,e_4$ is
  \[A=
   \left(
  \begin{array}{cccc}
    1&0&0&0 \\
    0&1&0&0 \\
    0&\frac{1-x_{31}}{\beta}&1&0 \\
     0&0&0&1
  \end{array}
 \right).
  \]
Since it becomes $AD= D_{13}A$, $D$ and $D_{13}$ are conjugate in $\mathrm{Aut}(\mathfrak{n})$. 
\end{proof}
From the above discussion, the following is satisfied.
\begin{prop}
The representation matrix of $D\in\delta^+(\mathfrak{n})$ which satisfies (C-1-3), concerning with $e_1,e_2,e_3,e_4$ is conjugate with the following.
\begin{align*}
  \left(
  \begin{array}{cccc}
     1 &-\beta &0&0 \\
     \beta&1&0&0 \\
     1 &0&1&0 \\
     0&0&0&2
  \end{array}
 \right),
 \end{align*}
 where,\ $0<\beta$.
\end{prop}
 {\bf{ (C-1-4)}}Case of the representation matrix of $D$ is given as follows.
  \begin{align*}
  D=
  \left(
  \begin{array}{cccc}
     \lambda_1 &0 &0&0 \\
     0&\lambda_2&0&0 \\
     0&x_{32}&\lambda_2&0 \\
     0&0&0&\lambda_1+\lambda_2
  \end{array}
 \right),\tag{4.9}
 \end{align*}
where $x_{32}\neq 0, \lambda_1< \lambda_2$.
 \begin{lem}
 For any $D\in \delta^+(\mathfrak{n})$ which satisfies (4.9), there exists unique $D_{14}\in \delta^+(\mathfrak{n})$, that is, conjugate to $D$ by $\mathrm{Aut}(\mathfrak{n})$ such that the representation matrix of $D_{14}$ concerning with $e_1,e_2,e_3,e_4$ takes the following form.
   \begin{align*}
  D_{14}=
  \left(
  \begin{array}{cccc}  
    \lambda_1 &0 &0&0 \\
     0&\lambda_2&0&0 \\
     0&1&\lambda_2&0 \\
     0&0&0&\lambda_1+\lambda_2
  \end{array}
 \right).
 \end{align*}
\end{lem}
\begin{proof}
  We take an element of $\mathrm{Aut}(\mathfrak{n})$ whose representation matrix concerning with $e_1,e_2,e_3,e_4$ is
  \[A=
   \left(
  \begin{array}{cccc}
     \lambda_1 &0 &0&0 \\
     0&\lambda_2&0&0 \\
     0 &0&\frac{\lambda_2}{x_{32}}&0 \\
     0&0&0&\lambda_1\lambda_2
  \end{array}
 \right).
  \]
Since it becomes $AD= D_{14}A$, $D$ and $D_{14}$ are conjugate in $\mathrm{Aut}(\mathfrak{n})$. 
\end{proof}
From the above discussion, the following is satisfied.
\begin{prop}
The representation matrix of $D\in\delta^+(\mathfrak{n})$ which satisfies (C-1-4), concerning with $e_1,e_2,e_3,e_4$ is conjugate with the following.
\begin{align*}
  \left(
  \begin{array}{cccc}
     x &0 &0&0 \\
     0&1&0&0 \\
     0 &1&1&0 \\
     0&0&0&1+x
  \end{array}
 \right),
 \end{align*}
 where,\ $0<x\leq 1$.\
\end{prop}
{\bf{ (C-1-5)}}Case of the representation matrix of $D$ is given as follows.
  \begin{align*}
  D=
  \left(
  \begin{array}{cccc}
     \alpha&-\beta &0&0 \\
     \beta&\alpha&0&0 \\
     0&x_{32}&\alpha&0 \\
     0&0&0&2\alpha
  \end{array}
 \right),\tag{4.10}
 \end{align*}
where $x_{32}\neq 0$.
 \begin{lem}
 For any $D\in \delta^+(\mathfrak{n})$ which satisfies (4.10), there exists unique $D_{15}\in \delta^+(\mathfrak{n})$, that is, conjugate to $D$ by $\mathrm{Aut}(\mathfrak{n})$ such that the representation matrix of $D_{15}$ concerning with $e_1,e_2,e_3,e_4$ takes the following form.
   \begin{align*}
  D_{15}=
  \left(
  \begin{array}{cccc}
     \alpha&-\beta &0&0 \\
     \beta&\alpha&0&0 \\
     0&1&\alpha&0 \\
     0&0&0&2\alpha
  \end{array}
 \right).
 \end{align*}
\end{lem}
\begin{proof}
  We take an element of $\mathrm{Aut}(\mathfrak{n})$ whose representation matrix concerning with $e_1,e_2,e_3,e_4$ is
  \[A=
   \left(
  \begin{array}{cccc}
    1&0&0&0 \\
    0&1&0&0 \\
    \frac{x_{32}-1}{\beta}&0&1&0 \\
     0&0&0&1
  \end{array}
 \right).
  \]
Since it becomes $AD= D_{15}A$, $D$ and $D_{15}$ are conjugate in $\mathrm{Aut}(\mathfrak{n})$. 
\end{proof}
From the above discussion, the following is satisfied.
\begin{prop}
The representation matrix of $D\in\delta^+(\mathfrak{n})$ which satisfies (C-1-5), concerning with $e_1,e_2,e_3,e_4$ is conjugate with the following.
\begin{align*}
  \left(
  \begin{array}{cccc}
   1 &-\beta &0&0 \\
     \beta&1&0&0 \\
     0 &1&1&0 \\
     0&0&0&2
  \end{array}
 \right),
 \end{align*}
 where,\ $0<\beta$.\
\end{prop}
 {\bf{ (C-1-6)}}Case of the representation matrix of $D$ is given as follows.
  \begin{align*}
  D=
  \left(
  \begin{array}{cccc}
     \lambda&1 &0&0 \\
     0&\lambda&0&0 \\
     x_{31}&0&\lambda&0 \\
     0&0&0&2\lambda
  \end{array}
 \right),\tag{4.11}
 \end{align*}
where, $x_{31}\neq0$.
 \begin{lem}
 For any $D\in \delta^+(\mathfrak{n})$ which satisfies (4.11), there exists unique $D_{16}\in \delta^+(\mathfrak{n})$, that is, conjugate to $D$ by $\mathrm{Aut}(\mathfrak{n})$ such that the representation matrix of $D_{16}$ concerning with $e_1,e_2,e_3,e_4$ takes the following form.
   \begin{align*}
  D_{16}=
  \left(
  \begin{array}{cccc}
     \lambda&1 &0&0 \\
     0&\lambda&0&0 \\
     1 &0&\lambda&0 \\
     0&0&0&2\lambda
  \end{array}
 \right).
 \end{align*}
\end{lem}
\begin{proof}
We take an element of $\mathrm{Aut}(\mathfrak{n})$ whose representation matrix concerning with $e_1,e_2,e_3,e_4$ is
  \[A=
   \left(
  \begin{array}{cccc}
     1&0 &0&0 \\
     0&1&0&0 \\
     0 &0&\frac{1}{x_{31}}&0 \\
     0&0&0&1
  \end{array}
 \right).
  \]
Since it becomes $AD= D_{16}A$, $D$ and $D_{16}$ are conjugate in $\mathrm{Aut}(\mathfrak{n})$. 
\end{proof}
From the above discussion, the following is satisfied.
\begin{prop}
The representation matrix of $D\in\delta^+(\mathfrak{n})$ which satisfies (C-1-6), concerning with $e_1,e_2,e_3,e_4$ is conjugate with the following.
\begin{align*}
  \left(
  \begin{array}{cccc}
     1 &1 &0&0 \\
     0&1&0&0 \\
     1 &0&1&0 \\
     0&0&0&2
  \end{array}
 \right).
 \end{align*}
\end{prop}
 {\bf{ (C-1-7)}}Case of the representation matrix of $D$ is given as follows.
  \begin{align*}
  D=
  \left(
  \begin{array}{cccc}
     \lambda&1 &0&0 \\
     0&\lambda&0&0 \\
     0&x_{32}&\lambda&0 \\
     0&0&0&2\lambda
  \end{array}
 \right),\tag{4.12}
 \end{align*}
where, $x_{32}\neq0$.
 \begin{lem}
 For any $D\in \delta^+(\mathfrak{n})$ which satisfies (4.12), there exists unique $D_{17}\in \delta^+(\mathfrak{n})$, that is, conjugate to $D$ by $\mathrm{Aut}(\mathfrak{n})$ such that the representation matrix of $D_{17}$ concerning with $e_1,e_2,e_3,e_4$ takes the following form.
   \begin{align*}
  D_1=
  \left(
  \begin{array}{cccc}
     \lambda&1 &0&0 \\
     0&\lambda&0&0 \\
     0&1&\lambda&0 \\
     0&0&0&2\lambda
  \end{array}
 \right).
 \end{align*}
\end{lem}
\begin{proof}
   We take an element of $\mathrm{Aut}(\mathfrak{n})$ whose representation matrix concerning with $e_1,e_2,e_3,e_4$ is
  \[A=
   \left(
  \begin{array}{cccc}
     1&0 &0&0 \\
     0&1&0&0 \\
     0 &0&\frac{1}{x_{32}}&0 \\
     0&0&0&1
  \end{array}
 \right).
  \]
 Since it becomes $AD= D_{17}A$, $D$ and $D_{17}$ are conjugate in $\mathrm{Aut}(\mathfrak{n})$. 
\end{proof}
From the above discussion, the following is satisfied.
\begin{prop}
The representation matrix of $D\in\delta^+(\mathfrak{n})$ which satisfies (C-1-7), concerning with $e_1,e_2,e_3,e_4$ is conjugate with the following.
\begin{align*}
  \left(
  \begin{array}{cccc}
     1 &1 &0&0 \\
     0&1&0&0 \\
     0&1&1&0 \\
     0&0&0&2
  \end{array}
 \right).
 \end{align*}
\end{prop}
 {\bf{ (C-1-8)}}Case of the representation matrix of $D$ is given as follows.
  \begin{align*}
  D=
  \left(
  \begin{array}{cccc}
     \lambda&1 &0&0 \\
     0&\lambda&0&0 \\
     x_{31}&x_{32}&\lambda&0 \\
     0&0&0&2\lambda
  \end{array}
 \right),\tag{4.13}
 \end{align*}
where, $x_{31},x_{32}\neq0$.
 \begin{lem}
 For any $D\in \delta^+(\mathfrak{n})$ which satisfies (4.13), there exists unique $D_{18}\in \delta^+(\mathfrak{n})$, that is, conjugate to $D$ by $\mathrm{Aut}(\mathfrak{n})$ such that the representation matrix of $D_{18}$ concerning with $e_1,e_2,e_3,e_4$ takes the following form.
   \begin{align*}
  D_{18}=
  \left(
  \begin{array}{cccc}
     \lambda&1 &0&0 \\
     0&\lambda&0&0 \\
     1&1&\lambda&0 \\
     0&0&0&2\lambda
  \end{array}
 \right).
 \end{align*}
\end{lem}
\begin{proof}
  We take an element of $\mathrm{Aut}(\mathfrak{n})$ whose representation matrix concerning with $e_1,e_2,e_3,e_4$ is
  \[A=
   \left(
  \begin{array}{cccc}
     1&\frac{x_{32}}{x_{31}}-1 &0&0 \\
     0&1&0&0 \\
     0 &0&\frac{1}{x_{31}}&0 \\
     0&0&0&1
  \end{array}
 \right),
  \]
Since it becomes $AD= D_{18}A$, $D$ and $D_{18}$ are conjugate in $\mathrm{Aut}(\mathfrak{n})$. 
\end{proof}
From the above discussion, the following is satisfied.
\begin{prop}
The representation matrix of $D\in\delta^+(\mathfrak{n})$ which satisfies (C-1-8), concerning with $e_1,e_2,e_3,e_4$ is conjugate with the following.
\begin{align*}
  \left(
  \begin{array}{cccc}
     1 &1 &0&0 \\
     0&1&0&0 \\
     1&1&1&0 \\
     0&0&0&2
  \end{array}
 \right).
 \end{align*}
\end{prop}
   {\bf{ (C-2)}} Case of $x_{33}=x_{11}+x_{22}$ and $x_{43}\neq0.$
\begin{lem}
  For any $D\in \delta^+(\mathfrak{n})$ which satisfies (C-2), there exists unique $D_1\in \delta^+(\mathfrak{n})$, that is, conjugate to $D$ by $\mathrm{Aut}(\mathfrak{n})$ such that the representation matrix of $D_1$ concerning with $e_1,e_2,e_3,e_4$ takes the following form.
  \begin{align}
  D_2=
  \left(
  \begin{array}{cccc}
     x_{11} &x_{12} &0&0 \\
     x_{21} &x_{22} &0&0 \\
     x_{31} &x_{32} &x_{11}+x_{22}&0 \\
     0&0&x_{43}&x_{11}+x_{22}
  \end{array}
 \right).\tag{4.14}
 \end{align}
\end{lem}
\begin{proof}
We take an element of $\mathrm{Aut}(\mathfrak{n})$ whose representation matrix concerning with $e_1,e_2,e_3,e_4$ is
  \[A=
   \left(
  \begin{array}{cccc}
     1&0&0&0 \\
     0&1&0&0 \\
     y_1&y_2&1&0 \\
     y_3&y_4&0&1
  \end{array}
 \right),
  \]
where
\[
\Tilde{X}=
  \left(
  \begin{array}{cc}
     x_{11} &x_{12} \\
     x_{21} &x_{22} 
  \end{array}
 \right).
  \]
Here, $y_1, y_2, y_3 \in \mathbb{R}$ are
 \begin{align*}
 y_1&=\frac{x_{11}x_{31}+x_{21}x_{32}}{\det\Tilde{X}},\\
 y_2&=\frac{x_{12}x_{31}+x_{22}x_{32}}{\det\Tilde{X}},\\
y_3&=\frac{x_{11}(X_{41}-x_{43}y_1)+x_{21}(x_{42}-x_{43}y_2)}{\det\Tilde{X}},\\
y_4&=\frac{x_{12}(X_{41}-x_{43}y_1)+x_{22}(x_{42}-x_{43}y_2)}{\det\Tilde{X}}.
 \end{align*}
Since it becomes $AD_2= D_{21}A$, $D$ and $D_{11}$ are conjugate in $\mathrm{Aut}(\mathfrak{n})$. 
\end{proof}
Furthermore, since transformations by a constant multiple are conjugate, we set $x_{43}=1$ below. Considering it in the same way as Proposition \ref{prop:3b}, the following holds.
\begin{prop}
The representation matrix of $D\in\delta^+(\mathfrak{n})$ which satisfies (C-2), concerning with $e_1,e_2,e_3,e_4$ is conjugate to one of the following.
\begin{align*}
  \left(
  \begin{array}{cccc}
     x &0&0&0 \\
     0 &1 &0&0 \\
     0&0&x+1&0 \\
     0&0&1&x+1
  \end{array}
 \right),\ 
 \left(
  \begin{array}{cccc}
     \frac{1}{2} &1&0&0 \\
     0 & \frac{1}{2} &0&0 \\
     0&0&1&0 \\
     0&0&1&1
  \end{array}
 \right),\ 
 \left(
  \begin{array}{cccc}
     \frac{1}{2} &-\alpha&0&0 \\
     \alpha &\frac{1}{2}&0&0 \\
     0&0&1&0 \\
     0&0&1&1
  \end{array}
 \right),
 \end{align*}
 where, $0<\alpha,x$.
\end{prop}
From the above discussion, the following is satisfied.
 \begin{thm}
Let $\mathfrak{g}$ be an SNC-algebra of dimension five with non-abelian derived algebra $\mathfrak{n}=[\mathfrak{g},\mathfrak{g}]$, which has a basis $e_1,e_2,e_3,e_4$ with the relations
\[
[e_1,e_2]=e_4,\ \text{others are zero.}
\]
Then there exists $e_5\in \mathfrak{g}\backslash\mathfrak{n}$ whose bracket operations are given by one of the following.
\begin{enumerate}
\item[(1)]
$ [e_5,e_1]=e_1,\ [e_5,e_2] =x_1e_2,\ [e_5,e_3] =x_2e_3,\ [e_5,e_4]=(1+x_1)e_4, 0<x_1,x_2.$
\item[(2)]
$ [e_5,e_1]=e_1,\ [e_5,e_2] =e_1+e_2,\ [e_5,e_3] =xe_3,\ [e_5,e_4]=2e_4, 0<x.$
\item[(3)]
 $ [e_5,e_1]=e_1+\alpha e_2,\ [e_5,e_2] =-\alpha e_1+e_2,\ [e_5,e_3] =xe_3,\ [e_5,e_4]=2e_4, 0<x,\alpha$
\item[(4)]
$ [e_5,e_1]=e_1+e_3,\ [e_5,e_2] =xe_2,\ [e_5,e_3] =e_3,\ [e_5,e_4]=(1+x)e_4, 1\leq x.$
\item[(5)]
$ [e_5,e_1]=e_1+\alpha e_2+e_3,\ [e_5,e_2] =-\alpha e_1+e_2,\ [e_5,e_3] =e_3,\ [e_5,e_4]=2e_4, 0<\alpha.$
\item[(6)]
$ [e_5,e_1]=xe_1,\ [e_5,e_2] =e_2+e_3,\ [e_5,e_3] =e_3,\ [e_5,e_4]=(1+x)e_4, 0<x \leq 1.$
\item[(7)]
$ [e_5,e_1]=e_1+\alpha e_2,\ [e_5,e_2] =-\alpha e_1+e_2+e_3,\ [e_5,e_3] =e_3,\ [e_5,e_4]=2e_4, 0<\alpha.$
\item[(8)]
$ [e_5,e_1]=e_1+e_3,\ [e_5,e_2] =e_1+e_2,\ [e_5,e_3] =e_3,\ [e_5,e_4]=2e_4.$
\item[(9)]
$ [e_5,e_1]=e_1,\ [e_5,e_2] =e_1+e_2+e_3,\ [e_5,e_3] =e_3,\ [e_5,e_4]=2e_4.$
\item[(10)]
$ [e_5,e_1]=e_1+e_3,\ [e_5,e_2] =e_1+e_2+e_3,\ [e_5,e_3] =e_3,\ [e_5,e_4]=2e_4.$
\item[(11)]
$ [e_5,e_1]=xe_1,\ [e_5,e_2] =e_2,\ [e_5,e_3] =(x+1)e_3+e_4,\ [e_5,e_4]=(x+1)e_4, 0<x.$
\item[(12)]
$ [e_5,e_1]=\frac{1}{2}e_1,\ [e_5,e_2] =e_1+\frac{1}{2}e_2,\ [e_5,e_3] =e_3+e_4,\ [e_5,e_4]=e_4.$
\item[(13)]
$ [e_5,e_1]=\frac{1}{2}e_1+\alpha e_2,\ [e_5,e_2] =-\alpha e_1+\frac{1}{2}e_2,\ [e_5,e_3] =e_3+e_4,\ [e_5,e_4]=e_4.$
\end{enumerate}
Any two Lie algebras in different types or with different parameters are not isomorphic to each other.
\end{thm}
From the above, the classification of  SNC-algebras in dimension five has been completed.
\section{SNC-algebras and Ricci curvature tensors}
\subsection{Ricci curvature tensors of Lie groups with left invariant metric}

Let an $n$-dimensional Lie group $G$ with a left invariant metric $\langle \ ,\ \rangle$. We denote by $\mathfrak{g}$ the Lie algebra of $G$. Consider the Levi-Civita connection $\nabla$ as a bilinear map $\nabla :\mathfrak{g}\times \mathfrak{g} \to \mathfrak{g}$. Since, if we express the Levi-Civita connection $\nabla$ by separating it into its symmetric and skew-symmetric parts, it becomes as follows.
\[
\nabla_XY=U(X,Y)+\frac{1}{2}[X,Y] \tag{5.1}
\]
for $X,Y\in \mathfrak{g}$.
Here, $U:\mathfrak{g} \times \mathfrak{g} \rightarrow \mathfrak{g}$ is given as follows. 
\[
\langle U(X,Y),Z\rangle=\frac{1}{2}\langle X,[Z,Y]\rangle+\frac{1}{2}\langle Y,[Z,X]\rangle  \tag{5.2}
\]
for $X,Y,Z\in \mathfrak{g}$.
Let $\{e_1,\ ...\ ,e_m\}$ be an orthonormal basis of $\mathfrak{g}$, then the Ricci curvature tensor $\text{Ric}(X,Y)$ for $X,Y\in\mathfrak{g}$ is given by
\[
\text{Ric}(X,Y)=\sum_{i=1}^m\langle R(e_i,X)Y,e_i\rangle.
\]
Here,
 \begin{align*}
 R(e_i,X)Y&=\nabla_{e_i}\nabla_XY-\nabla_X\nabla_{e_i}Y-\nabla_{[e_i,X]}Y\\
 &=\nabla_{e_i}(U(X,Y)+\frac{1}{2}[X,Y])-\nabla_X(U(e_i,Y)+\frac{1}{2}[e_i,Y])\\
&\ \ \ \ -(U([e_i,X],Y)+\frac{1}{2}[[e_i,X],Y]).\\
 \end{align*}
Furthermore, since $X\langle Y,Z\rangle=\langle \nabla_XY,Z\rangle+\langle Y,\nabla_XZ\rangle$,
\begin{align*}
 &\langle\nabla_{e_i}(U(X,Y)+\frac{1}{2}[X,Y]),e_i\rangle\\
 &=e_i\langle(U(X,Y)+\frac{1}{2}[X,Y],e_i\rangle-\langle(U(X,Y)+\frac{1}{2}[X,Y],\nabla_{e_i}e_i\rangle\\
&=\frac{1}{2}e_i\{\langle X,[e_i,Y]\rangle+\langle Y,[e_i,X]\rangle+\langle e_i,[X,Y]\rangle\}-\langle(U(X,Y)+\frac{1}{2}[X,Y],U(e_i,e_i)\rangle\\
&=-\langle(U(X,Y)+\frac{1}{2}[X,Y],U(e_i,e_i)\rangle\\
&=-\langle(U(X,Y),U(e_i,e_i)\rangle-\frac{1}{2}\langle e_i,[[X,Y],e_i]\rangle.
 \end{align*}
\begin{align*}
 &-\langle \nabla_X(U(e_i,Y)+\frac{1}{2}[e_i,Y]),e_i\rangle \\
 &=-X\langle(U(e_i,Y)+\frac{1}{2}[e_i,Y],e_i\rangle+\langle(U(e_i,Y)+\frac{1}{2}[e_i,Y],\nabla_Xe_i\rangle\\
&=-X\langle e_i,[e_i,Y]\rangle+\langle U(e_i,Y)+\frac{1}{2}[e_i,Y],U(X,e_i)+\frac{1}{2}[X,e_i]\rangle\\
&=\langle U(e_i,Y)+\frac{1}{2}[e_i,Y],U(X,e_i)+\frac{1}{2}[X,e_i]\rangle\\
&=\langle U(e_i,Y),U(X,e_i)\rangle+\frac{1}{2}\langle U(e_i,Y),[X,e_i]\rangle +\frac{1}{2}\langle [e_i,Y],U(X,e_i)\rangle +\frac{1}{4}\langle [e_i,y],[X,e_i]\rangle\\
&=\langle U(e_i,Y),U(X,e_i)\rangle+\frac{1}{4}\langle e_i,[[X,e_i],Y]\rangle +\frac{1}{4}\langle Y,[[X,e_i],e_i]\rangle \\
&\ \ \ \ +\frac{1}{4}\langle X,[[e_i,Y],e_i]\rangle +\frac{1}{4}\langle e_i,[[e_i,Y],X]\rangle +\frac{1}{4}\langle [e_i,Y],[X,e_i]\rangle.
 \end{align*}
\begin{align*}
 -\langle \nabla_{[e_i,X]}Y,e_i\rangle &=-\langle U([e_i,X],Y),e_i\rangle -\frac{1}{2}\langle [[e_i,X],Y],e_i\rangle \\
 &=-\frac{1}{2}\langle [e_i,X],[e_i,Y]\rangle-\frac{1}{2}\langle Y,[e_i,[e_i,X]]\rangle -\frac{1}{2}\langle [[e_i,X],Y],e_i\rangle.
 \end{align*}
From the above, $\text{Ric}(X,Y)$ is expressed as follows.
\begin{align*}
\text{Ric}(X,Y)=&\sum_{i=1}^m-\langle(U(X,Y),U(e_i,e_i)\rangle-\frac{1}{2}\langle e_i,[[X,Y],e_i]\rangle+\langle U(e_i,Y),U(e_i,X)\rangle\\
&-\frac{3}{4}\langle [[e_i,X],Y],e_i\rangle-\frac{1}{4}\langle Y,[e_i,[e_i,X]]\rangle+\frac{1}{4}\langle X,[[e_i,Y],e_i]\rangle\\
&+\frac{1}{4}\langle e_i,[[e_i,Y],X]\rangle-\frac{3}{4}\langle [e_i,X],[e_i,Y]\rangle. \tag{5.3}
 \end{align*}
\subsection{Ricci curvature tensors of  four-dimensional SNC-algebras}
In this subsection, we calculate the Ricci curvature tensors of SNC-algebras classified in Theorem 3.7 and Theorem 3.13.

Case of Theorem 3.7(1)

We define the orthonormal basis $e_1, e_2, e_3, e_4$ of $\mathfrak{n}$ as follows.
\begin{center}
  $ [e_4,e_1]=xe_1,\ [e_4,e_2] =ye_2,\   [e_4,e_3]=e_3,$ others are zero,
\end{center}
where $0<x\le y\le 1$.

From (5.2), each $U(e_i,e_j)$ is given by the following.
\begin{center}
 $U(e_1,e_1)=xe_4,\ U(e_2,e_2)=ye_4,\ U(e_3,e_3)=e_4,$\\
 $U(e_1,e_4)=-\frac{1}{2}xe_1,\ U(e_2,e_4)=-\frac{1}{2}ye_2,\ U(e_3,e_4)=-\frac{1}{2}e_3,$\\
 $U(e_1,e_2)=U(e_1,e_3)=U(e_2,e_3)=U(e_4,e_4)=0.$
\end{center}
Therefore, from (5.3), Ricci curvatures with respect to the basis are given as follows.
\begin{align*}
&\text{Ric}(e_1,e_1)=-x(x+y+1),\quad \text{Ric}(e_2,e_2)=-y(x+y+1),\\
&\text{Ric}(e_3,e_3)=-(x+y+1),\quad \text{Ric}(e_4,e_4)=-(x^2+y^2+1),\\
&\text{Ric}(e_1,e_2)=\text{Ric}(e_1,e_3)=\text{Ric}(e_1,e_4)=\text{Ric}(e_2,e_3)=\text{Ric}(e_2,e_4)=\text{Ric}(e_3,e_4)=0.
 \end{align*}
In particular, $\text{Ric}(e_i,e_j)\le0$. Furthermore, when $x=y=1$,
\[
\text{Ric}(e_i,e_j)=-3\langle e_i,e_j\rangle
\]
holds. Thus, when $x=y=1$, the SNC-algebra gives an Einstein manifold.

Case of Theorem 3.7(2)

We define the orthonormal basis $e_1, e_2, e_3, e_4$ such that it satisfies the following.
\begin{center}
  $ [e_4,e_1]=ze_1,\ [e_4,e_2] =e_2,\   [e_4,e_3]=e_2+e_3,\ 0<z$, others are zero.
\end{center}
From (5.2), each $U(e_i,e_j)$ is given by the following.
\begin{center}
 $U(e_1,e_1)=ze_4,\ U(e_2,e_2)=e_4,\ U(e_3,e_3)=e_4,$\\
 $U(e_1,e_4)=-\frac{1}{2}ze_1,\ U(e_2,e_4)=-\frac{1}{2}e_2-\frac{1}{2}e_3,\ U(e_3,e_4)=-\frac{1}{2}e_3,\ U(e_2,e_3)=\frac{1}{2}e_4,$\\
 $U(e_1,e_2)=U(e_1,e_3)=U(e_4,e_4)=0.$
\end{center}
Therefore, from (5.3), Ricci curvatures with respect to the basis are given as follows.
\begin{align*}
&\text{Ric}(e_1,e_1)=-z(z+2),\  \text{Ric}(e_2,e_2)=-z-\frac{3}{2},\  \text{Ric}(e_3,e_3)=-z-\frac{5}{2},\\
&\text{Ric}(e_4,e_4)=-z^2-\frac{5}{2},\  \text{Ric}(e_2,e_3)=-\frac{1}{2}z-1,\\
&\text{Ric}(e_1,e_2)=\text{Ric}(e_1,e_3)=\text{Ric}(e_1,e_4)=\text{Ric}(e_2,e_4)=\text{Ric}(e_3,e_4)=0.
 \end{align*}
In particular, $\text{Ric}(e_i,e_j)\le0$. In this case, since $0<z$, no Einstein manifold exsists.

Case of Theorem 3.7(3)

We define the orthonormal basis $e_1, e_2, e_3, e_4$ such that it satisfies the following.
\begin{center}
  $ [e_4,e_1]=e_1,\ [e_4,e_2] =e_1+e_2,\   [e_4,e_3]=e_2+e_3$,\ others are zero.
\end{center}
From (5.2), each $U(e_i,e_j)$ is given by the following.
\begin{center}
 $U(e_1,e_1)=e_4,\ U(e_2,e_2)=e_4,\ U(e_3,e_3)=e_4,$\\
 $U(e_1,e_4)=-\frac{1}{2}e_1-\frac{1}{2}e_2,\ U(e_2,e_4)=-\frac{1}{2}e_2-\frac{1}{2}e_3,\ U(e_3,e_4)=-\frac{1}{2}e_3,\ $\\
 $U(e_1,e_2)=\frac{1}{2}e_4,\ U(e_2,e_3)=\frac{1}{2}e_4,\ U(e_1,e_3)=U(e_4,e_4)=0.$
\end{center}
Therefore, from (5.3), Ricci curvatures with respect to the basis are given as follows.
\begin{align*}
&\text{Ric}(e_1,e_1)=-\frac{5}{2},\ \text{Ric}(e_2,e_2)=-3,\ \text{Ric}(e_3,e_3)=-\frac{7}{2},\\
&\text{Ric}(e_4,e_4)=-4,\ \text{Ric}(e_1,e_2)=-\frac{3}{2},\ \text{Ric}(e_2,e_3)=-\frac{3}{2},\\
&\text{Ric}(e_1,e_3)=\text{Ric}(e_1,e_4)=\text{Ric}(e_2,e_4)=\text{Ric}(e_3,e_4)=0.
 \end{align*}
In particular, $\text{Ric}(e_i,e_j)\le0$. In this case, no Einstein manifold exsists.

Case of Theorem 3.7(4)

We define the orthonormal basis $e_1, e_2, e_3, e_4$ such that it satisfies the following.
\begin{center}
  $ [e_4,e_1]=\alpha e_1+\beta e_2,\ [e_4,e_2] =-\beta e_1+\alpha e_2,\   [e_4,e_3]=e_3,\ 0<\alpha,\ \beta$, others are zero.
\end{center}
From (5.2), each $U(e_i,e_j)$ is given by the following.
\begin{center}
 $U(e_1,e_1)=\alpha e_4,\ U(e_2,e_2)=\alpha e_4,\ U(e_3,e_3)=e_4,$\\
 $U(e_1,e_4)=-\frac{1}{2}\alpha e_1+\frac{1}{2}\beta e_2,\ U(e_2,e_4)=-\frac{1}{2}\beta e_1-\frac{1}{2}\alpha e_2,\ U(e_3,e_4)=-\frac{1}{2}e_3,$\\
 $U(e_1,e_2)=U(e_1,e_3)=U(e_2,e_3)=U(e_4,e_4)=0.$
\end{center}
Therefore, from (5.3), Ricci curvatures with respect to the basis are given as follows.
\begin{align*}
&\text{Ric}(e_1,e_1)=-2\alpha^2-\alpha,\ \text{Ric}(e_2,e_2)=-2\alpha^2-\alpha,\\
&\text{Ric}(e_3,e_3)=-2\alpha-1,\ \text{Ric}(e_4,e_4)=-2\alpha^2-1,\\
&\text{Ric}(e_1,e_2)=\text{Ric}(e_1,e_3)=\text{Ric}(e_1,e_4)=\text{Ric}(e_2,e_3)=\text{Ric}(e_2,e_4)=\text{Ric}(e_3,e_4)=0.
 \end{align*}
In particular, $\text{Ric}(e_i,e_j)\le0$. Furthermore, when $\alpha=1$,
\[\text{Ric}(e_i,e_j)=-3\langle e_i,e_j\rangle\]
holds. Thus, when $\alpha=1$, the SNC-algebra gives an Einstein manifold.

Case of Theorem 3.13(1)

We define the orthonormal basis $e_1, e_2, e_3, e_4$ such that it satisfies the following.
\begin{center}
  $[e_1,e_2]=e_3,\ [e_4,e_1]=(1-x)e_1,\ [e_4,e_2] =xe_2,\   [e_4,e_3]=e_3,\ 0<x\le \frac{1}{2}$, others are zero.
\end{center}
From (5.2), each $U(e_i,e_j)$ is given by the following.
\begin{center}
 $U(e_1,e_1)=(1-x)e_4,\ U(e_2,e_2)=xe_4,\ U(e_3,e_3)=e_4,$\\
 $U(e_1,e_3)=-\frac{1}{2}e_2,\ U(e_2,e_3)=\frac{1}{2}e_1,$ 
 $U(e_1,e_4)=-\frac{1}{2}(1-x)e_1,\ U(e_2,e_4)=-\frac{1}{2}xe_2,\ U(e_3,e_4)=-\frac{1}{2}e_3,$\\
 $U(e_1,e_2)=U(e_4,e_4)=0$.
\end{center}
Therefore, from (5.3), Ricci curvatures with respect to the basis are given as follows.
\begin{align*}
&\text{Ric}(e_1,e_1)=-2(1-x)-\frac{1}{2},\ \text{Ric}(e_2,e_2)=-2x-\frac{1}{2},\\
&\text{Ric}(e_3,e_3)=-\frac{3}{2},\ \text{Ric}(e_4,e_4)=-\{(1-x)^2+x^2+1\},\\
&\text{Ric}(e_1,e_2)=\text{Ric}(e_1,e_3)=\text{Ric}(e_1,e_4)=\text{Ric}(e_2,e_3)=\text{Ric}(e_2,e_4)=\text{Ric}(e_3,e_4)=0.
 \end{align*}
In particular, $\text{Ric}(e_i,e_j)\le0$. Furthermore, when $x=\frac{1}{2}$,
\[\text{Ric}(e_i,e_j)=-\frac{3}{2}\langle e_i,e_j\rangle\]
holds. Thus, when $x=\frac{1}{2}$, the SNC-algebra gives an Einstein manifold.

Case of Theorem 3.13(2)

We define the orthonormal basis $e_1, e_2, e_3, e_4$ such that it satisfies the following.
\begin{center}
  $[e_1,e_2]=e_3,\ [e_4,e_1]=\frac{1}{2}e_1,\ [e_4,e_2] =e_1+\frac{1}{2}e_2,\   [e_4,e_3]=e_3$,\ others are zero.
\end{center}
From (5.2), each $U(e_i,e_j)$ is given by the following.
\begin{center}
 $U(e_1,e_1)=\frac{1}{2}e_4,\ U(e_2,e_2)=\frac{1}{2}e_4,\ U(e_3,e_3)=e_4,$\\
 $U(e_1,e_2)=\frac{1}{2}e_4,\ U(e_1,e_3)=-\frac{1}{2}e_2,\ U(e_2,e_3)=\frac{1}{2}e_1,$
 $U(e_1,e_4)=-\frac{1}{4}e_1-\frac{2}e_2-,\ U(e_2,e_4)=-\frac{1}{4}e_2,\ U(e_3,e_4)=-\frac{1}{2}e_3,\ U(e_4,e_4)=0.$
\end{center}
Therefore, from (5.3), Ricci curvatures with respect to the basis are given as follows.
\begin{align*}
&\text{Ric}(e_1,e_1)=-1,\ \text{Ric}(e_2,e_2)=-2,\\
&\text{Ric}(e_3,e_3)=-\frac{3}{2},\ \text{Ric}(e_4,e_4)=-2,\ \text{Ric}(e_1,e_2)=-1,\\
&\text{Ric}(e_1,e_3)=\text{Ric}(e_1,e_4)=\text{Ric}(e_2,e_3)=\text{Ric}(e_2,e_4)=\text{Ric}(e_3,e_4)=0.
 \end{align*}
In particular, $\text{Ric}(e_i,e_j)\le0$. In this case, no Einstein manifold exsists.

Case of Theorem 3.13(3)

We define the orthonormal basis $e_1, e_2, e_3, e_4$ such that it satisfies the following.
\begin{center}
  $[e_1,e_2]=e_3,\ [e_4,e_1]=\frac{1}{2}e_1+\alpha e_2,\ [e_4,e_2] =-\alpha e_1+\frac{1}{2}e_2,\   [e_4,e_3]=e_3$,\ others are zero.
\end{center}
From (5.2), each $U(e_i,e_j)$ is given by the following.
\begin{center}
 $U(e_1,e_1)=\frac{1}{2}e_4,\ U(e_2,e_2)=\frac{1}{2}e_4,\ U(e_3,e_3)=e_4,$\\
 $U(e_1,e_3)=-\frac{1}{2}e_2,\ U(e_2,e_3)=\frac{1}{2}e_1,$
 $U(e_1,e_4)=-\frac{1}{4}e_1+\frac{1}{2}\alpha e_2,\ U(e_2,e_4)=-\frac{1}{2}\alpha e_1-\frac{1}{4}e_2,\ U(e_3,e_4)=-\frac{1}{2}e_3,$\\
 $U(e_1,e_2)=U(e_4,e_4)=0.$
\end{center}
Therefore, from (5.3), Ricci curvatures with respect to the basis are given as follows.
\begin{align*}
&\text{Ric}(e_1,e_1)=-\frac{3}{2},\  \text{Ric}(e_2,e_2)=-\frac{3}{2},\ \text{Ric}(e_3,e_3)=-\frac{3}{2},\ \text{Ric}(e_4,e_4)=-\frac{3}{2},\\
&\text{Ric}(e_1,e_2)=\text{Ric}(e_1,e_3)=\text{Ric}(e_1,e_4)=\text{Ric}(e_2,e_3)=\text{Ric}(e_2,e_4)=\text{Ric}(e_3,e_4)=0.
 \end{align*}
In particular, $\text{Ric}(e_i,e_j)\le0$. Furthermore, 
\[\text{Ric}(e_i,e_j)=-\frac{3}{2}\langle e_i,e_j\rangle\]
holds. Thus, the SNC-algebra gives an Einstein manifold.

From the above discussion, the following is satisfied.
\begin{thm}
For the following SNC-algebras of dimension four, the corresponding solvable Lie groups with left invariant metric are Einstein manifolds. 
\begin{enumerate}
\item[(1)]  $ [e_4,e_1]=e_1,\ [e_4,e_2] =e_2,\   [e_4,e_3]=e_3,$ others are zero,
\item[(2)]  $ [e_4,e_1]=e_1+\beta e_2,\ [e_4,e_2] =-\beta e_1+e_2,\   [e_4,e_3]=e_3,\ 0<\ \beta$, others are zero.
\item[(3)]   $[e_1,e_2]=e_3,\ [e_4,e_1]=\frac{1}{2}e_1,\ [e_4,e_2] =\frac{1}{2}e_2,\   [e_4,e_3]=e_3$, others are zero.
\item[(4)]  $[e_1,e_2]=e_3,\ [e_4,e_1]=\frac{1}{2}e_1+\alpha e_2,\ [e_4,e_2] =-\alpha e_1+\frac{1}{2}e_2,\   [e_4,e_3]=e_3,\ 0<\ \alpha$,\ others are zero.
\end{enumerate}
\end{thm}
\section{Four-dimensional SNC-algebras equippeded with the structures of symmetric spaces}
In this section, we analyze the details of  four and five-dimensional SNC-algebras which admit structures of symmetric spaces. First, we will introduce a theorem that shows the necessary and sufficient conditions for being the symmetric space of Heintze. Second, we analyze the cases that become symmetric spaces in the four-dimensional SNC-algebra.
\subsection{Theorem of Heintze}

The following theorem was demonstrated by Heintze\cite{h}.
\begin{thm}[Heintze\cite{h}]
Let $\mathfrak{g}$ be a solvable Lie algebra with inner product, suth that $K<0$. Then the following conditions are mutually equivalent. 
\begin{enumerate}
\item[(1)]$\nabla R=0$.
\item[(2)]
\begin{enumerate}
\item[a)]$\mathfrak{g}={A_0}+\mathfrak{a}_1+\mathfrak{a}_2$\ is an orthogonal decomposition with  $\mathfrak{g}'=\mathfrak{a}_1+\mathfrak{a}_2,\ [\mathfrak{g}',\mathfrak{g}']=\mathfrak{a}_2$\ and $[\mathfrak{g}',\mathfrak{a}_2]=0$,
\item[b)]if $\mathrm{ad}A|\mathfrak{g}'=D_0+S_0$\ is decomposed into its symmetric and skew-symmetric psrt,\ then $D_0|\mathfrak{a}_i=i\cdot \lambda\cdot id\ $for\ $i=1,2,$and $S_0$\ is a skew-symmetric derivation of $\mathfrak{g}'$,
\item[c)]if $Z_1,...,Z_t$ form an orthonormal basis of $\mathfrak{a}_2$ and if the skew-symmetric maps $J_i:\mathfrak{a}_1\to\mathfrak{a}_1, i=1,...,t, $are defined by
\begin{center}
$[X,Y]=2\lambda\sum_{i=1}^t\langle X,J_iY\rangle Z_i$ for all $X,Y\in\mathfrak{a}_1$
\end{center}
then:
\begin{enumerate}
\item[$\alpha$)]${J_i}^2=-id$(i.e. $J_i$ orthogonal),
\item[$\beta$)]$J_iJ_k=-J_kJ_i$ for $i\neq k$,
\item[$\gamma$)]$J_iJ_kX\in\{J_1X,...,J_tX\}$ for all $X\in \mathfrak{a}_1$ and $i\neq k$.
\end{enumerate}
\end{enumerate}
\end{enumerate}
\end{thm}
In particular, in the case of symmetric spaces, it is sufficient to satisfy condition (1), so we can examine which of the four-dimensional SNC-algebras classified above correspond to symmetric spaces.
\subsection{Case of the derived algebra is abelian}
Let $\mathfrak{g}$ be an SNC-algebra of dimension four with the abelian derived algebra. From the assumption, $A_0=e_4,\ \mathfrak{a}_1=\{e_1,e_2,e_3\},\ \mathfrak{a}_2=0,\ \mathfrak{g'}=\{e_1,e_2,e_3\}$. Therefore $\mathfrak{g}$ satisfies the conditions of Theorem 6.1(2)a). Also, since $\mathfrak{a}_2=0$, it is obvious to satisfy the condition of c). Therefore, it is sufficient to check whether only condition b) is satisfied. Under the assumption, there are the following two cases that satisfy condition b).

Case of Theorem 3.7(1) and the following is satisfied.
\[[e_4,e_1]=e_1,\ [e_4,e_2] =e_2,\   [e_4,e_3]=e_3.\tag{6.1}\] 

Case of Theorem 3.7(4) and the following is satisfied.
\[  [e_4,e_1]=e_1+\beta e_2,\ [e_4,e_2] =-\beta e_1+e_2,\ [e_4,e_3]=e_3,\ 0<\beta. \tag{6.2}\]
In the case of (6.1), $\mathrm{ad}e_4|\mathfrak{g}'=D_0$, and if we take $\lambda=1$, then $D_0=id$. Therefore, $\mathfrak{g}$ satisfies the condition of b). In the case of (6.2), $\mathrm{ad}e_4|\mathfrak{g}'=D_0+S_0$, and if we take $\lambda=1$, then $D_0=id$. Therefore, $\mathfrak{g}$ satisfies the condition of b). 
\subsection{Case of the derived algebra is non-abelian}
Let $\mathfrak{g}$ be an SNC-algebra of dimension four with the non-abelian derived algebra. From the assumption, $A_0=e_4,\ \mathfrak{a}_1=\{e_1,e_2\},\ \mathfrak{a}_2=e_3,\ \mathfrak{g'}=\mathfrak{a}_1+\mathfrak{a}_2,\ [\mathfrak{g'},\mathfrak{a}_2]=0$. Therefore, $\mathfrak{g}$ satisfies the conditions of Theorem 6.1(2)a). Under the assumption, there are the following two cases that satisfy condition b).

Case of Theorem 3.13(1) and the following is satisfied.
\[[e_4,e_1]=\frac{1}{2}e_1,\ [e_4,e_2] =\frac{1}{2}e_2,\   [e_4,e_3]=e_3\tag{6.3}\] 

Case of Theorem 3.13(3) and the following is satisfied.
\[  [e_4,e_1]=\frac{1}{2}e_1+\beta e_2,\ [e_4,e_2] =-\beta e_1+\frac{1}{2}e_2,\ [e_4,e_3]=e_3,\ 0<\beta \tag{6.4}\]

In the case of (6.3), $\mathrm{ad}e_4|\mathfrak{g}'=D_0$, and if we take$\lambda=\frac{1}{2}$, then $D_0|\mathfrak{a}_1=\frac{1}{2}id,\ D_0|\mathfrak{a}_2=2\cdot \frac{1}{2}id$. Therefore, $\mathfrak{g}$ satisfies the condition of b). In the case of (6.4), $\mathrm{ad}e_4|\mathfrak{g}'=D_0+S_0$, and if we take $\lambda=\frac{1}{2}$, then $D_0|\mathfrak{a}_1=\frac{1}{2}id,\ D_0|\mathfrak{a}_2=2\cdot \frac{1}{2}id$. Therefore, $\mathfrak{g}$ satisfies the condition of b). 

Finally, confirm that condition c) is satisfied. Since $\mathfrak{a}_1=\{e_1,e_2\},\ \mathfrak{a}_2=e_3$, the skew-symmetric map $J:\mathfrak{a}_1\to \mathfrak{a}_1$ is defined so that its representation matrix with respect to the basis $e_1, e_2$ satisfies the following.
\[
\left(
  \begin{array}{cc}
    0 &1 \\
   -1&0
   \end{array}
\right)
\]
Since $t=1$, it is trivial that $\mathfrak{g}$ satisfies $\beta)$ and $\gamma)$, and since $J^2=-id$, $\mathfrak{g}$ satisfies $\alpha)$. Therefore, (6.3) and (6.4) are symmetric spaces. Therefore, the following holds.
\begin{thm}
When the orthonormal basis $e_1,e_2,e_3,e_4$ of the four-dimensional SNC-algebra is defined to satisfy one of the following, it becomes a symmetric space.
\begin{enumerate}
\item[(1)]
$ [e_4,e_1]=e_1,\ [e_4,e_2] =e_2,\ [e_4,e_3] =e_3,\ $others are zero.
\item[(2)]
$ [e_4,e_1]=e_1+\beta e_2,\ [e_4,e_2] =-\beta e_1+e_2,\ [e_4,e_3]=e_3,\ 0<\beta$, others are zero.
\item[(3)]
$[e_4,e_1]=\frac{1}{2}e_1,\ [e_4,e_2] =\frac{1}{2}e_2,\   [e_4,e_3]=e_3, [e_1,e_2]=e_3$, others are zero.
\item[(4)]
$[e_4,e_1]=\frac{1}{2}e_1+\beta e_2,\ [e_4,e_2] =-\beta e_1+\frac{1}{2}e_2,\ [e_4,e_3]=e_3,\ 0<\beta$, others are zero.
\end{enumerate}
\end{thm}
From Proposition 4 of Heintze\cite{h}, the symmetric spaces of Theorem6.2.(1) and (2) are precisely the hyperbolic spaces $\mathbb{R}H^4$. Moreover, the symmetric spaces of Theorem6.2.(3) and (4) are precisely the hyperbolic spaces $\mathbb{C}H^2$. 
\subsection{Case of the five-dimensional SNC-algebra}
First, five-dimensional SNC-algebra of type (A) can be classified in the same way as in the four-dimensional case. Second, the orthonormal basis $e_1,e_2,e_3,e_4,e_5$ of the five-dimensional SNC-algebra $\mathfrak{g}$ of type (B), the following is satisfied.
\[ \mathfrak{g}=\{e_5\}+\mathfrak{a}_1+\mathfrak{a}_2,\ \mathfrak{a}_1=\{e_1, e_2\},\ \mathfrak{a}_2=\{e_3, e_4\}.\]
From the above, in the case of (B), condition of Theorem 6.1(2)a) is not satisfied. Last, the orthonormal basis $e_1,e_2,e_3,e_4,e_5$ of the five-dimensional SNC-algebra $\mathfrak{g}$ of type (C), the following is satisfied.
\[ \mathfrak{g}=\{e_5\}+\mathfrak{a}_1+\mathfrak{a}_2,\ \mathfrak{a}_1=\{e_1, e_2,e_3\},\ \mathfrak{a}_2=\{e_4\}.\]
There is no object in the five-dimensional SNC-algebra of type (C) that satisfies the condition (2) c). Therefore, the following holds.
\begin{thm}
When the orthonormal basis $e_1,e_2,e_3,e_4,e_5$ of the five-dimensional SNC-algebra is defined to satisfy one of the following, it becomes a symmetric space.
\begin{enumerate}
\item[(1)]
    $ [e_5,e_1]=e_1,\ [e_5,e_2] =e_2,\ [e_5,e_3] =e_3,\ [e_5,e_4]=e_4,\ $others are zero.
\item[(2)]
 $ [e_5,e_1]=e_1,\ [e_5,e_2] =e_2,\ [e_5,e_3] =e_3+\beta e_4,\ [e_5,e_4]=-\beta e_3+e_4, 0<\beta,\ $others are zero.
\item[(3)]
$ [e_5,e_1]=e_1+\beta e_2,\ [e_5,e_2] =-\beta e_1+e_2,\ [e_5,e_3] =e_3+\beta' e_4,\ [e_5,e_4]=-\beta' e_3+e_4,\ 0<\beta, \beta',\ $others are zero.
\end{enumerate}
\end{thm}
From Proposition 4 of Heintze\cite{h}, the symmetric spaces of Theorem6.3.(1), (2) and (3) are precisely the hyperbolic spaces $\mathbb{R}H^5$.
\subsection{Detailed analysis as a symmetric space}
From Example 2.1, it can be seen that the symmetric space given in Theorem6.2.(1) has constant curvature of $-1$. Next, we consider the symmetric space given in Theorem6.2.(2). From (5.1) and (5.2), the Riemann curvature tensor is given as follows.
\begin{align*}
 R(X,Y)Z&=\nabla_X\nabla_YZ-\nabla_Y\nabla_XZ-\nabla_{[X,Y]}Z\\
 &=\nabla_X(U(Y,Z)+\frac{1}{2}[Y,Z])-\nabla_Y(U(X,Z)+\frac{1}{2}[X,Z])\\
 &\ \ \ \ -(U([X,Y],Z)+\frac{1}{2}[[X,Y],Z])\\
 &=U(X,U(Y,Z))+\frac{1}{2}[X,U(Y,Z)]+\frac{1}{2}U(X,[Y,Z])+\frac{1}{4}[X,[Y,Z]]\\
 &\ \ \ \ -U(Y,U(X,Z))-\frac{1}{2}[Y,U(X,Z)]-\frac{1}{2}U(Y,[X,Z])-\frac{1}{4}[Y,[X,Z]]\\
 &\ \ \ \ -U([X,Y],Z)-\frac{1}{2}[[X,Y],Z].
 \end{align*}
Especially, $R(X,X)Y=0, R(X,Y)Z=-R(Y,X)Z$.

Riemann curvature tensors in the orthonormal basis $e_1, e_2, e_3, e_4$ that satisfies Theorem6.2.(2) is given by the following.
\begin{align*}
 R(e_1,e_2)e_3=R(e_2,e_3)e_1=R(e_3,e_1)e_2=0.\\
 R(e_1,e_2)e_4=R(e_2,e_4)e_1=R(e_4,e_1)e_2=0.\\
 R(e_1,e_4)e_3=R(e_4,e_3)e_1=R(e_3,e_1)e_4=0.\\
 R(e_4,e_2)e_3=R(e_2,e_3)e_4=R(e_3,e_4)e_2=0.\\
 R(e_1,e_2)e_1=e_2, R(e_1,e_3)e_1=e_3, R(e_1,e_4)e_1=e_4.\\
 R(e_2,e_1)e_2=e_1, R(e_2,e_3)e_2=e_3, R(e_2,e_4)e_2=e_4.\\
 R(e_3,e_1)e_3=e_1, R(e_3,e_2)e_3=e_2, R(e_3,e_4)e_3=e_4.\\
 R(e_4,e_1)e_4=e_1, R(e_4,e_2)e_4=e_2, R(e_4,e_3)e_4=e_3.
 \end{align*}
In particular, since $ R(X,Y)Z=-\{\langle Y,Z\rangle X-\langle X,Z\rangle Y\}$ is satisfied, it becomes a symmetric space of constant curvature $-1$. Moreover, symmetric space given in Theorem6.2.(1) and (2) are not isomorphic to each other as Lie algebras but their curvatures with respective to the uniquely determined orthnormal basis coincide.

Next, we consider the symmetric space given in Theorem6.2.(3) and (4). Rieman curvature tensors in the orthonormal basis $e_1, e_2, e_3, e_4$ that satisfies Theorem6.2.(3) is given by the following.
\begin{align*}
 R(e_1,e_2)e_1=e_2, R(e_1,e_3)e_1=\frac{1}{4}e_3, R(e_1,e_4)e_1=\frac{1}{4}e_4.\\
 R(e_2,e_1)e_2=e_1, R(e_2,e_3)e_2=\frac{1}{4}e_3, R(e_2,e_4)e_2=\frac{1}{4}e_4.\\
 R(e_3,e_1)e_3=\frac{1}{4}e_1, R(e_3,e_2)e_3=\frac{1}{4}e_2, R(e_3,e_4)e_3=e_4.\\
 R(e_4,e_1)e_4=\frac{1}{4}e_1, R(e_4,e_2)e_4=\frac{1}{4}e_2, R(e_4,e_3)e_4=e_4.\\
 R(e_1,e_2)e_3=-\frac{1}{2}e_4, R(e_2,e_3)e_1=\frac{1}{4}e_4, R(e_3,e_1)e_2=\frac{1}{4}e_4.\\
 R(e_1,e_2)e_4=\frac{1}{2}e_3, R(e_2,e_4)e_1=-\frac{1}{4}e_3, R(e_4,e_1)e_2=-\frac{1}{4}e_3.\\
 R(e_1,e_3)e_4=\frac{1}{4}e_2, R(e_3,e_4)e_1=-\frac{1}{2}e_2, R(e_4,e_1)e_3=-\frac{1}{4}e_2.\\
 R(e_2,e_3)e_4=-\frac{1}{4}e_1, R(e_3,e_4)e_2=\frac{1}{2}e_1, R(e_4,e_2)e_3=-\frac{1}{4}e_1.\\
 \end{align*}
Riemann curvature tensors in the orthonormal basis $e_1, e_2, e_3, e_4$ that satisfies Theorem6.2.(4) is given by the following.
\begin{align*}
 R(e_1,e_2)e_1=e_2, R(e_1,e_3)e_1=\frac{1}{4}e_3, R(e_1,e_4)e_1=\frac{1}{4}e_4.\\
 R(e_2,e_1)e_2=e_1, R(e_2,e_3)e_2=\frac{1}{4}e_3, R(e_2,e_4)e_2=\frac{1}{4}e_4.\\
 R(e_3,e_1)e_3=\frac{1}{4}e_1, R(e_3,e_2)e_3=\frac{1}{4}e_2, R(e_3,e_4)e_3=e_4.\\
 R(e_4,e_1)e_4=\frac{1}{4}e_1, R(e_4,e_2)e_4=\frac{1}{4}e_2, R(e_4,e_3)e_4=e_4.\\
 R(e_1,e_2)e_3=-\frac{1}{2}e_4, R(e_2,e_3)e_1=\frac{1}{4}e_4, R(e_3,e_1)e_2=\frac{1}{4}e_4.\\
 R(e_1,e_2)e_4=\frac{1}{2}e_3, R(e_2,e_4)e_1=-\frac{1}{4}e_3, R(e_4,e_1)e_2=-\frac{1}{4}e_3.\\
 R(e_1,e_3)e_4=\frac{1}{4}e_2, R(e_3,e_4)e_1=-\frac{1}{2}e_2, R(e_4,e_1)e_3=-\frac{1}{4}e_2.\\
 R(e_2,e_3)e_4=-\frac{1}{4}e_1, R(e_3,e_4)e_2=\frac{1}{2}e_1, R(e_4,e_2)e_3=-\frac{1}{4}e_1.\\
 \end{align*}
From this, the symmetric spaces given in Theorem6.2.(3) and (4) do not have constant curvature. Moreover, symmetric space given in Theorem6.2.(3) and (4) are not isomorphic to each other as Lie algebras but their curvatures with respective to the uniquely determined orthnormal basis coincide.
\\

\section*{Acknowledgements}
The author would like to express sincere gratitude to all those who contributed to this research. In particular, the author is deeply grateful to Professor Makiko Sumi Tanaka for her guidance and support throughout this study.

\end{document}